\documentclass[article,11pt]{amsart}

\usepackage[T1]{fontenc}
\usepackage[applemac]{inputenc}
\usepackage[francais,english]{babel}
\usepackage{amssymb}
\usepackage{amsmath}
\usepackage{latexsym}
\usepackage{marvosym}
\usepackage{amsfonts}
\usepackage{graphicx}
\usepackage{hyperref}
\usepackage{color}
\usepackage{setspace}
\usepackage{enumerate}
\hypersetup{
backref=true, 
pagebackref=true,
hyperindex=true, 
colorlinks=true, 
breaklinks=true, 
urlcolor= blue, 
linkcolor= blue, 
pdfauthor={Romain AZAIS}, 
}

\newcommand{\prob}{\mathbf{P}}
\newcommand{\esp}{\mathbf{E}}
\newcommand{\ud}{\text{\normalfont d}}

\newcommand{\fin}{
$\hfill
\mathbin{\vbox{\hrule\hbox{\vrule height1.5ex \kern.6em
\vrule height1.5ex}\hrule}}$}

\newtheorem{prop1}{Proposition}[section]
	
\newtheorem{lem1}[prop1]{Lemma}
\newtheorem{cor1}[prop1]{Corollary}
\newtheorem{theo}[prop1]{Theorem}

\newtheorem{hyps}[prop1]{Assumptions}
\newtheorem{rem}[prop1]{Remark}

\email{\vspace{-0.2cm}romain.azais@inria.fr}
\email{\vspace{-0.2cm}dufour@math.u-bordeaux1.fr}
\email{\vspace{-0.7cm}anne.petit@u-bordeaux2.fr}
\begin{document}

\title[Estimation of the distribution of the inter-jumping times for PDMP's]
{Nonparametric estimation of the conditional distribution of the inter-jumping times for piecewise-deterministic Markov processes}

\author{Romain Aza\"{\i}s}
\author{François Dufour}
\author{Anne Gégout-Petit}
\address{
INRIA Bordeaux Sud-Ouest, team CQFD, France and Universit\'e Bordeaux, IMB, CNRS UMR 5251,
351, Cours de la Libération, 33405 Talence cedex, France.
}

\thanks{This work was supported by ARPEGE program of the French National Agency of Research (ANR), project “FAUTOCOES”, number ANR-09-SEGI-004.}


\begin{abstract}
This paper presents a nonparametric method for estimating the conditional density associated to the jump rate of a piecewise-deterministic Markov process. In our framework, the estimation needs only one observation of the process within a long time interval. Our method relies on a generalization of Aalen's multiplicative intensity model. We prove the uniform consistency of our estimator, under some reasonable assumptions related to the primitive characteristics of the process. A simulation example illustrates the behavior of our estimator.
\end{abstract}

\maketitle

\section{Introduction}

This paper is devoted to the nonparametric estimation of the jump rate for piecewise-deterministic Markov processes, from only one observation of the process within a long time interval, under assumptions which ensure the ergodicity of an embedded chain. Our approach is based on methods investigated in the previous work of the authors \cite{Az12a} and a generalization of the well-known multiplicative intensity model, developed by Aalen in \cite{AalPHD,Aal77,Aal78} in the middle of the seventies.

Piecewise-deterministic Markov processes (PDMP's) have been introduced in the literature by Davis in \cite{Dav} as a general class of non-diffusion stochastic models. They are a family of Markov processes involving deterministic motion punctuated by random jumps, which occur either when the flow hits the boundary of the state space or in a Poisson-like fashion with nonhomogeneous rate. The path depends on three local characteristics namely the flow $\Phi$, the jump rate $\lambda$, which determines the interarrival times, and the transition kernel $Q$, which specifies the post-jump location. A suitable choice of the state space and the local characteristics $\Phi$, $\lambda$ and $Q$ provides stochastic models covering a large number of problems, for example in reliability (see \cite{Dav} and \cite{MR2528336}). Denote by $f$ the conditional density of the interarrival times associated to $\lambda$. This is a function of two variables: a spatial mark and time. The purpose of this paper is to develop a nonparametric procedure to estimate this function, when only one observation of the process within a long time is available. To the best of our knowledge, the nonparametric estimation of the conditional distribution of the interarrival times for this class of stochastic models has never been studied. Furthermore, this paper relies on \cite{Az12a} in which we focus on the nonparametric estimation of the jump rate and the cumulative rate for a class of non-homogeneous marked renewal processes. This class of stochastic models amounts to considering a particular piecewise-deterministic process, whose post-jump locations do not depend on interarrival times.

As counting processes may model a large variety of problems, nonparametric and semiparametric estimation methods have been developed by many authors for their statistical inference. The famous multiplicative intensity model has been extensively investigated by Aalen since 1975 (see \cite{AalPHD,Aal77,Aal78}). This model postulates the existence of a predictable process $Y$ and a deterministic function $\lambda$, called the jump rate or the hazard rate, such that the stochastic intensity of the underlying counting process $N$ is given by the product $Y\lambda$. In this context, Aalen provided a useful method for estimating the cumulative rate $\Lambda(t)=\int_0^t\lambda(s)\ud s$. The associated consistent estimator is now called the Nelson-Aalen estimator. In 1983, Ramlau-Hansen focused on smoothing this estimator by some kernel methods, in order to estimate directly the jump rate $\lambda$. He provided a nonparametric estimate of the rate $\lambda$ in \cite{Ram83}.

As mentioned before, a large number of estimation problems are related to the estimation of jump rates depending on both time and a spatial variable. The Nelson-Aalen estimator is proved to be well-adapted for a large variety of developments and applications (see the book \cite{And} and the references therein), in particular in survival analysis, or in statistics of processes. For example, one may apply Aalen's approach for estimating the jump rate of a marked counting process, whose state space is finite, from independent observations. More recently, Comte \textit{et al.} proposed in \cite{COMTE} a new strategy for the inference for counting processes in presence of covariates, under the multiplicative assumption.

Semiparametric estimation methods have been mainly investigated in presence of continuous covariates, beginning with Cox \cite{Cox72}. One may refer the reader to the book \cite{And} and the references therein for a large review of the literature on these models. There exists also an extensive literature on nonparametric approaches when the spatial mark takes its values in a continuous space. We do not attempt to present an exhaustive survey on this topic, but refer the interested reader to \cite{AalenHistory,And,MR676128,Martinussen2006} and the references therein for detailed expositions on these techniques. In particular, McKeague and Utikal proposed in 1990 a nonparametric estimator of the jump rate when the covariate belongs to $[0,1]$ (see \cite{McK90}). Their approach is based on a generalization of Aalen's multiplicative model. It consists in smoothing a Nelson-Aalen type estimator both in spatial and time directions. The authors demonstrated the uniform convergence in probability of their estimator. Li and Doss extended in turn McKeague and Utikal's work for the multidimensional case in \cite{MR1345201}. This paper relies on a local linear fit in the spatial direction. Their theoretical results concern the weak convergence of the proposed estimators. The interested reader may also consult the papers written by Utikal \cite{MR1208878,MR1445041}. These two papers deal with the nonparametric estimation of the jump rate for two special classes of marked counting processes, observed within a long time, under some continuous-time martingale assumptions. The Euclidean structure of the covariate state space plays a key role in the papers mentioned above. At the same time, the nonparametric approach has been considered by Beran \cite{BERAN}, Stute \cite{MR840519} and Dabrowska \cite{MR932943}, but in the independent and identically distributed case.

An inherent difficulty throughout the paper is related to the presence of forced jumps, when the process reaches the boundary of the state space. This feature has been introduced by Davis in \cite{Dav} and is very attractive to model lots applications. For instance, one may refer the reader to a capacity expansion model introduced by Davis \textit{et al.} in \cite{Dav2}. In \cite{subtilin}, the authors develop a PDMP to capture the mechanism behind antibiotic released by bacteria \textit{B. subtilis}. Forced jumps are used to model a deterministic switching in the mode when the concentration of nutrients rises over a certain threshold. For applications in reliability, one may find in \cite{MR1679540} an example of shock models with failures of threshold type. An other application may be found in \cite{DeS}, where the authors focus on the optimal stopping for a PDMP modeling the state of a metallic structure subject to corrosion. In the book \cite{Jacobsen}, the authors derive likelihood processes for observation of PDMP's without forced jumps. For this class of models without boundary jumps, this approach could lead to an estimation procedure in the parametric case. We choose an other procedure which is well adapted in our nonparametric framework and for PDMP's with forced jumps.

Our approach relies on our previous work \cite{Az12a} and a generalization of Aalen's multiplicative model. The main difficulty is related to the dependence of the transition kernel on the previous interarrival time. This excludes the techniques developed in the literature \cite{Az12a,MR1345201,McK90,MR1208878,MR1445041} for estimating the jump rate $\lambda$. The keystone of the present paper is to consider the Markov chain $(Z_n,Z_{n+1},S_{n+1})_{n\geq0}$, where the $Z_n$'s denote the post-jump locations of the process, and the $S_n$'s denote the interarrival times. The main idea in this work is to deal with the conditional distribution of $S_{n+1}$ given $Z_n$ and $Z_{n+1}$. We prove that this conditional distribution admits a jump rate $\widetilde{\lambda}$, under some regularity conditions on the primitive data of the process (see Proposition \ref{txdesaut}). Furthermore, a conditional independence result is satisfied by the discrete-time process $(Z_n,Z_{n+1},S_{n+1})_{n\geq0}$ (see Proposition \ref{mozart2}). In this context, we focus on the estimation of the jump rate $\widetilde{\lambda}$, which is a function of three variables: two spatial marks and time. Moreover, the two spatial variables take their values on a general metric space. In particular, this rules out the procedures investigated by the authors mentioned above \cite{MR1345201,McK90,MR1208878,MR1445041}. As a consequence, our method consists in involving a thin partition of the state space. We take a leaf out of a few proofs of our previous work \cite{Az12a} for estimating the function $\widetilde{l}(A,B,t)$, which is an approximation of the jump rate $\widetilde{\lambda}(x,y,t)$, for $x\in A$ and $y\in B$. In the rest of the paper, we use the convergence property of this estimator to tackle the estimation problem of the density of interest $f$. We study the connection between $\widetilde{\lambda}$ and the conditional density $f$ (see Proposition \ref{approxf}). This link involves the conditional probability $\prob_{\nu}(S_1>t,Z_1\in B|Z_0\in A)$, where $\nu$ denotes the invariant measure of the process $(Z_n)_{n\geq0}$. An efficient estimate of this quantity is presented in Proposition \ref{unifps}. We provide a nonparametric estimator of the conditional density $f$, and we prove a result of uniform convergence in Theorem \ref{pdmp:theo:fin}. We ensure the consistency of our estimator, under interpretable and reasonable conditions related to the primitive characteristics of the process.


The paper is organized as follows. Section \ref{sec:problem} is devoted to the precise formulation of our problem. We first recall, in Subsection \ref{sub:def_pdmp}, the definition of a PDMP. In Subsection \ref{sub:results}, we provide the assumptions that we need in the sequel. Next, in Subsection \ref{sub:main}, we give the main steps in the estimation of the conditional density of interest (see Propositions \ref{approxf}, \ref{estimationltilde} and \ref{unifps}). Our main result of consistency is presented in Theorem \ref{pdmp:theo:fin}. In Section \ref{sec:preliminary}, we derive some preliminary results. We introduce a new process in Subsection \ref{sub:lambdatilde} and we focus on the existence of the conditional jump rate $\widetilde{\lambda}$ of this Markov process (see Proposition \ref{txdesaut}). In Subsection \ref{sub:ergo}, we state some ergodicity results, which will be useful in order to prove the uniform consistency of our estimator. Section \ref{sec:proofs} is devoted to most of the proofs of our major results given in Subsection \ref{sub:results}. Finally, a numerical example is given in Section \ref{s:simu} for illustrating the good behavior of our estimator. The technical proof of a conditional independence result (Proposition \ref{mozart2}) is deferred in Appendix \ref{appendixA}.

\section{Problem formulation}
\label{sec:problem}

This section is devoted to the definition of a piecewise-deterministic Markov process. Moreover, we present also the required assumptions and our main results.

\subsection{Definition of a PDMP}
\label{sub:def_pdmp}

Here, we focus on the definition of a piecewise-deterministic Markov process on a separable metric space. The process evolves in an open subset $E$ of a separable metric space $(\mathcal{E},d)$. The motion is defined by the three local characteristics $(\lambda,Q,\Phi)$.
\begin{itemize}
	\item $\Phi : E\times \mathbf{R}_+ \to \overline{E}$ is the deterministic flow. It satisfies,
		$$\forall \xi \in E,~\forall s,t \geq 0,~ \Phi(\xi,t+s) = \Phi(   \Phi(\xi,t) , s ) .$$
		For each $\xi\in E$, $t^\star(\xi)$ denotes the deterministic exit time from $E$:
		$$ t^\star(\xi) = \inf \{t>0~:~\Phi(\xi,t) \in \partial E\},$$
		with the usual convention $\inf\emptyset = +\infty$.
	\item $\lambda: \overline{E}\to\mathbf{R}_+$ is the jump rate. It is a measurable function which satisfies,
		$$\forall\xi\in E,~\exists \varepsilon>0,~ \int_0^\varepsilon \lambda\big( \Phi(\xi,s) \big) \ud s < +\infty .$$
	\item $Q$ is a Markov kernel on $(\overline{E},\mathcal{B}(\overline{E}))$ which satisfies,
		$$\forall \xi\in\overline{E},~Q(\xi,\overline{E}\setminus\{\xi\})=1 \quad\text{and}\quad Q(\xi,E)=1.$$
\end{itemize}

\noindent
There exists a filtered probability space $(\Omega,\mathcal{A},(\mathcal{F}_t)_{t\geq0},\prob_{\nu_0})$, on which the process $(X_t)_{t\geq 0}$ is defined (see \cite{Dav}). The probability distribution of the initial value $X_0$ is $\nu_0$. Starting from $x\in E$, the motion can be described as follows. $T_1$ is a positive random variable whose survival function is,
$$ \forall t \geq 0, ~ \mathbf{P}_{\nu_0}( T_1 > t|X_0=x) = \exp \left(  - \int_0^t \lambda( \Phi(\xi,s) ) \ud s \right) \mathbf{1}_{\{ 0 \leq t < t^\star(x) \}} .$$
This jump time occurs either when the flow reaches the boundary of the state space $E$ at time $t^\star(x)$ or in a Poisson-like fashion with rate $\lambda$ before. One chooses an $E$-valued random variable $Z_1$ according to the distribution $Q(\Phi(\xi,T_1),\cdot)$. Let us remark that the post-jump location depends on the interarrival time $T_1$. The trajectory between the times $0$ and $T_1$ is given by
\begin{displaymath}
X_t = \left\{ 
\begin{array}{cl}
\Phi(x,t) 	& \text{for $0\leq t < T_1$,} \\
Z_1		& \text{for $t=T_1$.}
\end{array}
\right.
\end{displaymath}
Now, starting from $X_{T_1}$, one selects the time $S_2 = T_2-T_1$ and the post-jump location $Z_2$ in a similar way as before, and so on. This gives a strong Markov process with the $T_k$'s as the jump times (with $T_0=0$). One often considers the embedded Markov chain $(Z_n,S_n)_{n\geq0}$ associated to the process $(X_t)_{t\geq 0}$ with $Z_n=X_{T_n}$, $S_n=T_n-T_{n-1}$ and $S_0=0$, that is, the $Z_n$'s denote the post-jump locations of the process, and the $S_n$'s denote the interarrival times.

On the strength of \cite{Dav} (Chapter 1, Section 24 Definition of the PDP), the embedded chain $(Z_n,S_n)_{n\geq0}$ is generated by a stochastic dynamic system. There exist two measurable functions $\varphi$ and $\psi$, and two independent random sequences $(\varepsilon_n)_{n\geq0}$ and $(\delta_n)_{n\geq0}$, such that, for any $n\geq1$,
\begin{equation} \label{pdmp:dyn2}
\left\{
\begin{array}{l l l}
S_n &=& \varphi(Z_{n-1},\delta_{n-1}) ,\\
Z_n &=& \psi(Z_{n-1},S_n,\varepsilon_{n-1}) .
\end{array}
\right.
\end{equation}
%

\noindent
For a matter of readability, we introduce the following notations,
\begin{displaymath}
\begin{array}{lllccc}
\forall\xi\in E,&\forall t\geq0,&~&\overline{\lambda}(\xi,t) &=& \lambda\big(\Phi(\xi,t)\big) ,\\
\forall\xi\in E,&\forall t\geq0,&\forall A\in\mathcal{B}(E),&\overline{Q}(\xi,t,A) &=&Q\big(\Phi(\xi,t),A) .
\end{array} 
\end{displaymath}
In all the sequel, let us denote by $f$ and $G$ the probability density function and the survival function associated to the jump rate $\overline{\lambda}$.
\begin{eqnarray*}
\forall\xi\in E,~\forall t\geq0,~G(\xi,t) &=&\exp\left(-\int_0^t\overline{\lambda}(\xi,s)\ud s\right) ,\\
\forall\xi\in E,~\forall t\geq0,~f(\xi,t)&=&\overline{\lambda}(\xi,t)G(\xi,t) .
\end{eqnarray*}
The purpose of this paper is to provide a nonparametric estimate of the conditional density $f$. We shall prove the consistency of the estimator from some assumptions related to the characteristics of the process.

\subsection{Assumptions}
\label{sub:results}

		In this part, we present all the assumptions that we need in the sequel. They may be classified into two parts: the assumptions on the transition kernel, and some assumptions of regularity. Let us denote by $\mathcal{C}_1$ the following set,
		$$\mathcal{C}_1 = \left\{B\in\mathcal{B}(E)~:~\stackrel{\circ}{B}\neq\emptyset\right\}.$$
		In addition, $\mathcal{C}_2$ is defined by
		$$\mathcal{C}_2 = \left\{A\in\mathcal{C}_1~:~A\,\text{relatively compact set such that $\overline{A}\cap\partial E = \emptyset$}\right\} .$$
		
		\begin{hyps} \label{hyps1}
		\textbf{Assumptions on the transition kernel.}
		\vspace{-0.3cm}
		\begin{enumerate}[a)]
		
			\item	The transition kernel $\overline{Q}$ may be written in the following way,
		$$ \forall \xi\in E,~\forall s\geq0,~\forall B\in\mathcal{B}(E),~\overline{Q}(\xi,s,B) = \int_B \widetilde{Q}(\xi,s,y) \mu(\ud y),$$
		where $\mu$ is an auxiliary measure on $(E,\mathcal{B}(E))$ such that, for any measurable set $B$ with non-empty interior, $\mu(B)>0$.

			\item For any $x,y\in E$, $\widetilde{Q}(x,\cdot,y)$ is a continuous function.

			\item There exists $m>0$ such that,		
				\begin{equation}	\label{minQ}
					\forall x\in E,~\forall s\geq0,~\forall y\in E,~\widetilde{Q}(x,s,y) \geq m .
				\end{equation}	

			\item 
			For any $e>0$, there exist $p\geq 2$ and a family of non-overlapping sets $\{B_1,\dots,B_p\}$, $B_k\in\mathcal{C}_1$ for any $k$, such that
			$$\max_{1\leq k\leq p} \text{\normalfont diam}~\!B_k<e \qquad \text{and} \qquad \forall x\in E,~\forall t\geq0,~\sum_{k=1}^p \overline{Q}(x,t,B_k) = 1.$$
			Without loss of generality, $B_1$ is assumed to be the set with the largest diameter.
			
		\end{enumerate}
		\end{hyps}

		\noindent
		If $E$ is bounded, the last point of Assumptions \ref{hyps1} is obviously satisfied. For the sake of clarity, let us introduce the following notations,
		\begin{eqnarray}
		\forall \xi\in E,	~\forall s\geq0,~\forall y\in E,~	G\widetilde{Q}(\xi,s,y) 	&=&		G(\xi,s)\widetilde{Q}(\xi,s,y) ,		\label{pdmp:def:GQtilde}\\
		\forall \xi \in E,	~\forall s\geq0,~\forall y\in E,~	f\widetilde{Q}(\xi,s,y) 	&=& 	f(\xi,s)\widetilde{Q}(\xi,s,y) .			\nonumber
		\end{eqnarray}
		Furthermore, we denote also,
		\begin{eqnarray}
		\forall \xi\in E,~\forall s\geq0,~\forall B\in\mathcal{B}(E),~G\overline{Q}(\xi,s,B) 	&=&	G(\xi,s)\overline{Q}(\xi,s,B) , \label{pdmp:def:GQbarre}\\
		\forall \xi\in E,~\forall s\geq0,~\forall B\in\mathcal{B}(E),~f\overline{Q}(\xi,s,B) 	&=& 	f(\xi,s)\overline{Q}(\xi,s,B) . \label{pdmp:def:fQbarre}
		\end{eqnarray}

		\begin{hyps} \label{hyps2}
		\textbf{Assumptions of regularity.} 
		\vspace{-0.3cm}
		\begin{enumerate}[a)]
		
			\item $f(\xi,\cdot)$ is a strictly positive and continuous function for any $\xi\in E$. 				
			
			\item There exists a locally integrable function $M:\mathbf{R}_+\to\mathbf{R}_+$ such that,		
			$$\forall \xi\in E,~\forall t\geq0,~\overline{\lambda}(\xi,t)\leq M(t) .$$
		
			\item 
			$t^\star:E\to\mathbf{R}_+$ is a bounded function.
			
			\item There exists a constant $[f\widetilde{Q}]_{Lip}>0$ such that, for any $x,y,u,v\in E$ and for any $0\leq t<t^\star(x)\wedge t^\star(u)$, 	 
			\begin{equation}
				\Big| f\widetilde{Q}(x,t,y) - f\widetilde{Q}(u,t,v)\Big| \leq [f\widetilde{Q}]_{Lip} d_2\big((x,y),(u,v)\big) .	\label{fQlip}
			\end{equation}
		
			\item There exists a constant $[f]_{Lip}>0$ such that,		
				$$\forall x,y \in E,~\forall t\geq0,~\big|f(x,t)-f(y,t)\big| \leq [f]_{Lip} d(x,y) .$$


			\item There exists a constant $[G\widetilde{Q}]_{Lip}>0$ such that, for any $x,y,u,v\in E$,		
			\begin{equation}
				\big|G\widetilde{Q}(x,t^\star(x),y) - G\widetilde{Q}(u,t^\star(u),v)\big| \leq [G\widetilde{Q}]_{Lip} d_2\big((x,y),(u,v)\big) .
				\label{GQlip}
			\end{equation}

		\end{enumerate}
		\end{hyps}		

\subsection{Main results}

\label{sub:main}

		In this subsection, we consider a set $A\in\mathcal{C}_2$ and $e>0$. This induces the existence of $p\geq 2$ and a family $\{B_1,\dots,B_p\}$ depending on $e$ and satisfying the conditions of Assumptions \ref{hyps1}. First of all, we provide a technical but required lemma.		
		
		\begin{lem1}
		We have $\displaystyle\inf_{\xi\in A} t^\star(\xi)>0$. $t^\star(A)$ denotes this quantity.
		\end{lem1}

		\begin{proof}
		One may refer to the proof of Lemma 4.7 given in \cite{Az12a}.
		\end{proof}

Our strategy for estimating the conditional probability density function $f$ consists in the introduction of two functions $\widetilde{l}(A,B_k,t)$ and $\widetilde{H}(A,B_k,t)$, for $0\leq t<t^\star(A)$. They provide a way to approximate the density of interest $f(\xi,t)$ for $\xi\in A$. Indeed, $f(\xi,t)$ is close to
$$\sum_{k=1}^p \widetilde{l}(A,B_k,t) \widetilde{H}(A,B_k,t)$$
(see Proposition \ref{approxf}). The function $\widetilde{H}$ will be defined in $(\ref{eq:def:Htilde})$, while the definition of $\widetilde{l}$ will be provided in Lemma \ref{defltilde}. One may give an interpretation of these two functions. First, $\widetilde{l}(A,B_k,t)$ is an approximation of the jump rate $\widetilde{\lambda}(x,y,t)$ of $S_{n+1}$ given $Z_n=x$ and $Z_{n+1}=y$ under the stationary regime at time $t$, for $x$ in $A$ and $y$ in $B_k$. The existence of the function $\widetilde{\lambda}$ is established in Subsection \ref{sub:lambdatilde}. Roughly speaking, the quantity $\widetilde{l}(A,B_k,t)$ may be seen as the jump rate from $A$ to $B_k$ at time $t$ and under the stationary regime. Furthermore, we will establish that $\widetilde{H}(A,B_k,t)$ is exactly the conditional probability $\prob_{\nu}(S_1>t,Z_1\in B_k | Z_0\in A)$ where $\nu$ denotes the limit distribution of the embedded Markov chain $(Z_n)_{n\geq0}$. In this context, a natural estimator of $f(\xi,t)$ is given by
$$\widehat{f}_n(A,t) =  \sum_{k=1}^p \widehat{\widetilde{l}}_{n}(A,B_k,t)\, \widehat{p}_n(A,B_k,t) ,$$
where $\widehat{\widetilde{l}}_n(A,B_k,t)$ estimates $\widetilde{l}(A,B_k,t)$, while $\widehat{p}_n(A,B_k,t)$ is an estimator of $\widetilde{H}(A,B_k,t)$. Both these estimators will be given in $(\ref{estimateurltilde})$ and $(\ref{pdmp:def:estpn})$.

Now, we present our main results in a more precise context. The proofs of most of the results provided in this subsection are deferred into Section \ref{sec:proofs}. One considers the distance $d_2$ on $\mathcal{E}^2$ induced by the distance $d$ and the Manhattan norm on $\mathbf{R}^2$. For any $(x,y),(u,v)\in\mathcal{E}^2$,
$$ d_2\big( (x,y) , (u,v)\big) =d(x,u)+d(y,v) .$$
$\text{\normalfont diam}_2~\!A\times B_k$ denotes the diameter of $A\times B_k\subset\mathcal{E}^2$ with respect to this distance.

		\begin{prop1} \label{approxf}
		Let $\xi\in A$. For any $0\leq t<t^\star(A)$,
		$$ \left|f(\xi,t) - \sum_{k=1}^p \widetilde{l}(A,B_k,t) \widetilde{H}(A,B_k,t)\right| \leq \text{Cst}_1\,\text{\normalfont diam}~\!A + \text{Cst}_2\,\text{\normalfont diam}_2\,A\times B_1.$$
		\end{prop1}
		
		\begin{proof}
		Subsection \ref{sub:approxf} is dedicated to this proof.
		\end{proof}

		

		%
		On the one hand, we focus on the estimation of $\widetilde{l}(A,B_k,t)$. We consider the continuous-time processes $N_n(A,B_k,\cdot)$ and $Y_n(A,B_k,\cdot)$ defined on $[0,t^\star(A)[$ by
		\begin{equation}
		\label{Nn}
		N_n(A,B_k,t) =  \sum_{i=0}^{n-1} \mathbf{1}_{\{S_{i+1}\leq t\}}\mathbf{1}_{\{Z_i\in A\}}\mathbf{1}_{\{Z_{i+1}\in B_k\}} ,
		\end{equation}
		and
		\begin{equation}
		\label{yn}
		Y_n(A,B_k,t) =  \sum_{i=0}^{n-1} \mathbf{1}_{\{S_{i+1}\geq t\}}\mathbf{1}_{\{Z_i\in A\}}\mathbf{1}_{\{Z_{i+1}\in B_k\}} .
		\end{equation}
		$Y_n(A,B_k,t)^+$ denotes the generalized inverse of $Y_n(A,B_k,t)$. It is given by
		\begin{equation}	\label{pdmp:def:ynplus}
		Y_n(A,B_k,t)^+ = 
		\left\{
		\begin{array}{cll}
			0 &\text{if}&Y_n(A,B_k,t) =0 ,\\
			\frac{1}{Y_n(A,B_k,t)} &\text{else.}&
		\end{array}
		\right.
		\end{equation}		
		The estimator of the function $\widetilde{l}(A,B_k,\cdot)$ that we provide is obtained by kernel methods. Let $K$ be a continuous kernel whose support is $[-1,1]$, $b>0$ and $0<t<t^\star(A)$. Our estimator is defined, for any time $s$ between $0$ and $t$ by
		\begin{equation}
		\label{estimateurltilde}
		\widehat{\widetilde{l}}_{n,b,t}(A,B_k,s) = \frac{1}{b} \sum_{k=0}^{n-1} K\left( \frac{s - S_{i+1}}{b}\right) Y_n(A,B_k,S_{i+1})^+\mathbf{1}_{\{S_{i+1}\leq t\}}\mathbf{1}_{\{Z_i\in A\}} \mathbf{1}_{\{Z_{i+1}\in B_k\}} .
		\end{equation}
		We have the following result of convergence.
		\begin{prop1}				\label{estimationltilde}																	
		Let $0<r_1<r_2<t<t^\star(A)$ and $x\in E$. There exists a sequence $(\beta_n(A,B_k))_{n\geq0}$ (depending on $t$) which almost surely tends to $0$, such that
		$$ \sup_{r_1\leq s\leq r_2} \left|\widehat{\widetilde{l}}_{n,\beta_n(A,B_k),t}(A,B_k,s) -  \widetilde{l}(A,B_k,s)\right| \stackrel{\prob_x~}{\longrightarrow} 0~\text{when $n\to+\infty$}.$$
		\end{prop1}
		
		\begin{proof}
		The reader may find in Subsection \ref{sub:ltilde} some properties of $\widetilde{l}$ and the proof of this proposition.
		\end{proof}

		%
		On the other hand, we estimate the quantity $\widetilde{H}(A,B_k,t)$ by its empirical version $\widehat{p}_n(A,B_k,t)$ given by
		\begin{equation} \label{pdmp:def:estpn}
		\widehat{p}_n(A,B_k,t) = \frac{  \sum_{i=0}^{n-1}  \mathbf{1}_{\{S_{i+1}>t\}} \mathbf{1}_{\{Z_{i+1}\in B_k\}} \mathbf{1}_{\{Z_i\in A\}} }{ \sum_{i=0}^{n-1}  \mathbf{1}_{\{Z_i\in A\}} } .
		\end{equation}
		We shall prove the following result of consistency.
		\begin{prop1}	\label{unifps}
		Let $0<t<t^\star(A)$. For any $x\in E$, we have when $n$ goes to infinity,
		$$
\sup_{0\leq s\leq t} \left| \widehat{p}_n(A,B_k,s)- \widetilde{H}(A,B_k,t) \right|  \to 0 ~\prob_{x}\textit{-a.s.}
		$$
		\end{prop1}
		
		\begin{proof}
		The proof may be found in Subsection \ref{sub:Htilde}.
		\end{proof}

		\noindent
		In the light of foregoing, we estimate $f(\xi,s)$, with $\xi\in A$ and $0\leq s\leq t<t^\star(A)$,by
		$$\widehat{f}_n(A,s) =  \sum_{k=1}^p \widehat{\widetilde{l}}_{n,\beta_n(A,B_k),t}(A,B_k,s)\, \widehat{p}_n(A,B_k,s) .$$
		The dependence in $t$ and $(\beta_n(A,B_k))_{n\geq0}$ is implicit for the sake of clarity. Our main result of convergence lies in the following theorem.

\begin{theo}
\label{pdmp:theo:fin}
Let $\mathcal{K}$ be a compact subset of $E$ and $\xi\in E$. For any $\epsilon,\eta>0$, there exist an integer $N$ and a finite partition $P=(A_l)$ of $\mathcal{K}$ such that, for any $0<t<\min_l t^\star(A_l)$, there exists for each couple $(l,k)$, a sequence $(\beta_n(A_l,\beta_k))_{n\geq0}$ (depending on $t$), which almost surely tends to $0$, such that for any $n\geq N$, for any $0<r_1<r_2<t$,
$$
\prob_{\xi} \left( \sup_{x\in\mathcal{K}}\sup_{r_1\leq s\leq r_2} \left| \sum_{l} \widehat{f}_{n}(A_l,t)  \mathbf{1}_{\{x\in A_l\}} - f(x,s) \right| > \eta\right) <\epsilon .
$$
\end{theo}

\begin{proof}
Let $(A_l)$ a partition of $\mathcal{K}$. Let us fix $l$. Using Proposition \ref{estimationltilde} and Proposition \ref{unifps}, there exists a family of sequences $\{(\beta_n(A_l,B_k))_{n\geq0}\}_{1\leq k\leq p}$, such that
$$
\sup_{r_1\leq s\leq r_2} \left| \sum_{k=1}^p \widehat{\widetilde{l}}_{n,\beta_n(A_l,B_k),t}(A_l,B_k,s) \widehat{p}_n(A_l,B_k,s) - \sum_{k=1}^p \widetilde{l}(A_l,B_k,s) \widetilde{H}(A_l,B_k,s)\right| \stackrel{\prob_{\xi}~}{\longrightarrow} 0 ,
$$
when $n$ goes to infinity. In addition, on the strength of Proposition \ref{approxf}, for any $x\in\mathcal{K}$, the distance
$$
\sup_{r_1\leq s\leq r_2}\left|f(x,s) - \sum_{l=1}^{|P|}\mathbf{1}_{\{x\in A_l\}} \sum_{k=1}^p \widetilde{l}(A_l,B_k,s) \widetilde{H}(A_l,B_k,s) \right|
$$
is arbitrarily small, since one may choose a thin enough partition $(A_l)$ and thanks to Assumptions \ref{hyps1}. The triangle inequality immediately ensures the result.
\end{proof} 

\noindent
In the previous theorem, if the compact subset $\mathcal{K}$ is close to the state space $E$, then the lower bound of the $t^\star(A_k)$'s is small, for any partition $(A_k)$. Therefore, one estimates the conditional density $f$ on a large part of $E$, but within a small time interval. Conversely, if $\mathcal{K}$ is chosen centered in $E$, one may estimate $f$ on a small part of the state space, but within a long time.

\section{Preliminary discussion}
\label{sec:preliminary}

\subsection{Conditional distribution of $S_{n+1}$ given $Z_n,Z_{n+1}$}
\label{sub:lambdatilde}

We shall state that the conditional distribution of $S_{n+1}$ given $Z_n$ and $Z_{n+1}$ admits a jump rate (see Proposition \ref{txdesaut}). This result directly induces a corollary about the compensator of the counting process $N^{n+1}$, given for any $t\geq0$ by
$$N^{n+1}(t) = \mathbf{1}_{\{S_{n+1}\leq t\}},$$
in a particular filtration (see Corollary \ref{cumparsita}). First, we shall establish that $G(x,t^\star(x))$ and $G\widetilde{Q}(x,t^\star(x),y)$ are two strictly positive numbers for any $x,y\in E$.
		
		\begin{rem}Let $x,y\in E$. On the one hand, we have
		\begin{equation}													\label{pdmp:minoG}
		G\big(x,t^\star(x)\big) \geq \exp\left( - \int_0^{t^\star(x)} M(s)\ud s \right) > 0 ,
		\end{equation}
		because $M$ is a locally integrable function. On the other hand, from equations $(\ref{minQ})$ and $(\ref{pdmp:def:GQtilde})$, we have
		\begin{equation}
		G\widetilde{Q}(x,t^\star(x),y) \geq m G\big(x,t^\star(x)\big) > 0 .					\label{pdmp:minoGQ}
		\end{equation}
		\end{rem}
	
	\noindent
	Let us denote by $R$ the transition kernel of the Markov chain $(Z_n)_{n\geq0}$. In the following remark, we especially give an explicit formula for $R$.
		
		\begin{rem}							\label{pdmp:rem:R}
		The kernel $R$ may be written in the following way \cite[(34.12), page 116]{Dav}, for any $x\in E$, $B\in\mathcal{B}(E)$,
		\begin{equation}		\label{pdmp:exprR}
		R(x,B) = \int_0^{t^\star(x)} f\overline{Q}(x,s,B) \ud s + G\overline{Q}(x,t^\star(x),B),
		\end{equation}
		where the functions $G\overline{Q}$ and $f\overline{Q}$ have already been defined by $(\ref{pdmp:def:GQbarre})$ and $(\ref{pdmp:def:fQbarre})$. Furthermore, since the transition kernel $\overline{Q}$ admits $\widetilde{Q}$ as a density according to the measure $\mu$, one may also write
		\begin{equation}		\label{pdmp:exprR2}
		R(x,B) = \int_B \Bigg[\int_0^{t^\star(x)} f\widetilde{Q}(x,s,y) \ud s + G\widetilde{Q}(x,t^\star(x),y)\Bigg]\mu(\ud y).
		\end{equation}
		Together with $(\ref{pdmp:minoG})$ and $(\ref{pdmp:minoGQ})$,
		\begin{eqnarray}					
		R(x,B)&\geq&G\big(x,t^\star(x)\big) \int_B \widetilde{Q}(x,t^\star(x),y)\mu(\ud y) \nonumber \\
		&\geq&m \mu(B)  \exp\left( - \int_0^{t^\star(x)} M(s)\ud s\right) . 	\label{pdmp:minoR}
		\end{eqnarray}
		In particular, if $\stackrel{\circ}{B}\neq\emptyset$, $R(x,B)>0$ because $\mu(B)>0$ according to Assumptions \ref{hyps1}.
		\end{rem}

		\noindent
		Let us denote by $\nu_n$ (resp. by $\widetilde{\nu}_n$, by $\eta_n$) the distribution of $Z_n$ (resp. of the couple $(Z_n,Z_{n+1})$, of $(Z_n,Z_{n+1},S_{n+1})$). The following remark deals with the relation between $\widetilde{\nu}_n$, $\nu_n$ and the transition kernel $R$.

		\begin{rem}
		Let $n$ be an integer and $A\times B\in\mathcal{B}(E)^{\otimes2}$. One may write
		\begin{eqnarray}
		\widetilde{\nu}_n (A\times B) &=& \prob_{\nu_0}(Z_{n+1}\in B , Z_n \in A) 		\nonumber \\
		&=& \int_A R(x,B) \nu_n(\ud x) .										\label{pdmp:tildenu}
		\end{eqnarray}
		\end{rem}
		
		\noindent
		We focus our attention on the relation between the probability measures $\eta_n$ and $\nu_n$.
		\begin{rem}
		Let $t\geq0$ and $B\in\mathcal{B}(E)$. We have
		\begin{align*}
		\prob_{\nu_0}(S_{n+1}&>t,Z_{n+1}\in B|Z_n)\\
		&=\mathbf{1}_{\{0\leq t<t^\star(Z_n)\}}\Bigg[\int_{t\wedge t^\star(Z_n)}^{t^\star(Z_n)} f\overline{Q}(Z_n,s,B)\ud s+G\overline{Q}(Z_n,t^\star(Z_n),B) \Bigg] .
		\end{align*}
		This obviously induces that, for any $A,B\in\mathcal{B}(E)$ and $t\geq0$,
		\begin{align}
		\eta_n\big(A&\times B\times]t,+\infty[\big) \nonumber \\
		&= \int_A \mathbf{1}_{\{0\leq t<t^\star(x)\}}\Bigg[\int_{t\wedge t^\star(x)}^{t^\star(x)} f\overline{Q}(x, s,B)\ud s+G\overline{Q}(x,t^\star(x),B) \Bigg] \nu_n(\ud x) . \label{expr:etan}
		\end{align}
		\end{rem}

	\noindent
	The main result of this part lies in the following proposition. It deals with the existence of a jump rate for the conditional distribution of $S_{n+1}$ given $Z_n, Z_{n+1}$.

\begin{prop1}
\label{txdesaut}
Let $n$ be an integer. The conditional distribution of $S_{n+1}$ given $Z_n,Z_{n+1}$ satisfies, for any $t\geq0$,
$$
\prob_{\nu_0}(S_{n+1}>t|Z_n,Z_{n+1})=\exp\left(-\int_0^{t\wedge t^\star(Z_n)}\widetilde{\lambda}(Z_n,Z_{n+1},s)\ud s\right)\mathbf{1}_{\{0\leq t<t^\star(Z_n)\}} ,
$$
where the jump rate $\widetilde{\lambda}$ is defined for any $x,y\in E$, $0\leq t\leq t^\star(x)$, by
\begin{equation}
\label{defLAMBDA}
\widetilde{\lambda}(x,y,t) = \frac{f\widetilde{Q}(x,t,y)}{\int_t^{t^\star(x)} f\widetilde{Q}(x,s,y)\ud s + G\widetilde{Q}(x,t^\star(x),y)} .
\end{equation}
\end{prop1}

\begin{proof}
Let $x,y\in E$. $\widetilde{\lambda}(x,y,\cdot)$ is a continuous function on the interval $[0,t^\star(x)]$ because $f(x,\cdot)$ and $\widetilde{Q}(x,\cdot,y)$ are two continuous functions in the light of Assumptions \ref{hyps2}. The survival function $\widetilde{G}$ associated to $\widetilde{\lambda}$ is defined for any $x,y\in E$, $0\leq t\leq t^\star(x)$, by
		\begin{equation}\label{defG001}
		\widetilde{G}(x,y,t) =\exp\left(-\int_0^t\widetilde{\lambda}(x,y,s)\ud s\right).
		\end{equation}
		Moreover, for any $0\leq t<t^\star(x)$, $\widetilde{\lambda}(x,y,t)=-u'(t)/u(t)$ with
		$$u(t)=\int_t^{t^\star(x)} f\widetilde{Q}(x, s,y) \ud s +G\widetilde{Q}(x,t^\star(x),y) .$$
		As a consequence, we have
		\begin{align*}
		\int_0^t\widetilde{\lambda}(x,y,s)\ud s~&= -\ln\left(\int_t^{t^\star(x)} f\widetilde{Q}(x, s,y) \ud s+G\widetilde{Q}(x,t^\star(x),y)\right)\\
		&~\phantom{=} +\ln\left(\int_0^{t^\star(x)} f\widetilde{Q}(x, s,y) \ud s +G\widetilde{Q}(x,t^\star(x),y)\right) \\
		&=~-\ln\left( \frac{\int_t^{t^\star(x)} f\widetilde{Q}(x, s,y) \ud s +G\widetilde{Q}(x,t^\star(x),y)}{\int_0^{t^\star(x)} f\widetilde{Q}(x, s,y) \ud s +G\widetilde{Q}(x,t^\star(x),y)}\right).
		\end{align*}
		Finally, together with $(\ref{defG001})$,
		\begin{equation} \label{defG}
		\widetilde{G}(x,y,t) = \frac{\int_t^{t^\star(x)} f\widetilde{Q}(x, s,y) \ud s  +  G\widetilde{Q}(x,t^\star(x),y)}{   \int_0^{t^\star(x)} f\widetilde{Q}(x,s,y) 
\ud s  + G \widetilde{Q}(x,t^\star(x),y)}  .
		\end{equation}
		In order to establish the expected result, we need to prove the equality
		\begin{equation}
		\eta_n\big(A\times B\times]t,+\infty[\big) = \int_{A\times B} \widetilde{G}\big(x,y,t\wedge t^\star(x)\big)\mathbf{1}_{\{0\leq t<t^\star(x)\}} \widetilde{\nu}_n(\ud x\times\ud y) ,	\label{pdmp:desint:toshow}
		\end{equation}
		for any $A,B\in\mathcal{B}(E)$ and $t\geq0$. By $(\ref{pdmp:tildenu})$, we have
		\begin{align*}
		\int_{A\times B} \widetilde{G}\big(x,y,t\wedge t^\star(x)\big)\mathbf{1}&_{\{0\leq t<t^\star(x)\}} \widetilde{\nu}_n(\ud x\times\ud y)\\
		&= \int_{A\times B} \widetilde{G}\big(x,y,t\wedge t^\star(x)\big)\mathbf{1}_{\{0\leq t<t^\star(x)\}} R(x,\ud y) \nu_n(\ud x).
		\end{align*}
		Thus, with $(\ref{pdmp:exprR2})$ and $(\ref{defG})$, we obtain
		\begin{align*}
		\int_{A\times B}& \widetilde{G}\big(x,y,t\wedge t^\star(x)\big)\mathbf{1}_{\{0\leq t<t^\star(x)\}} \widetilde{\nu}_n(\ud x\times\ud y)\\
		&= \int_{A\times B} \mathbf{1}_{\{0\leq t<t^\star(x)\}} \Bigg[\int_{t\wedge t^\star(x)}^{t^\star(x)} f\widetilde{Q}(x, s,y) \ud s  +  G\widetilde{Q}(x,t^\star(x),y) \Bigg] \mu(\ud y) \nu_n(\ud x)\\
		&= \int_A \mathbf{1}_{\{0\leq t<t^\star(x)\}} \Bigg[\int_{t\wedge t^\star(x)}^{t^\star(x)} f\overline{Q}(x, s,B) \ud s  +  G\overline{Q}(x,t^\star(x),B) \Bigg] \nu_n(\ud x). 
		\end{align*}
		Together with the expression $(\ref{expr:etan})$ of $\eta_n$, this directly implies $(\ref{pdmp:desint:toshow})$ and, therefore, the expected result.
		\end{proof}

		\begin{rem} \label{pdmp:rem:desint}
		Let $n$ be an integer. On the strength of Jirina's theorem (see for instance Theorem 11.7 of \cite{Ouv}), there exists a kernel family $(\gamma_{x,y}(\cdot))_{(x,y)\in E^2}$ such that for any $A\times B\times\Gamma$ in $\mathcal{B}(E)^{\otimes 2}\otimes\mathcal{B}(\mathbf{R}_+)$, we have
		$$ \eta_n(A\times B\times \Gamma) = \int_{A\times B} \gamma_{x,y}(\Gamma) \widetilde{\nu}_n(\ud x\times\ud y) .$$
		Let $x,y\in E^2$. By Proposition \ref{txdesaut}, $\gamma_{x,y}$ does not depend on $n$. Furthermore, we have the relation
		\begin{eqnarray*}
		\gamma_{x,y}\big([t,+\infty[\big)	&=& \prob_{\nu_0}(S_1>t | Z_0=x,Z_1=y) \\
								&=& \widetilde{G}(x,y,t),
		\end{eqnarray*}
		if $t<t^\star(x)$, by $(\ref{pdmp:desint:toshow})$.
		\end{rem}		
	
		\noindent	
		We also have the following continuous-time martingale property.
		\begin{cor1} \label{cumparsita}
		Let $i$ be an integer. The continuous-time process $M^{i+1}$ given by,
		\begin{equation}	\label{pdmp:def:Miplus1}
		\forall 0\leq t<t^\star(Z_i),~M^{i+1}(t)    =    N^{i+1}(t)   -   \int_0^t   \widetilde{\lambda}(Z_i,Z_{i+1},u) \mathbf{1}_{\{S_{i+1}\geq u\}} \ud u ,
		\end{equation}
		is a $(\sigma(Z_i,Z_{i+1})\vee\mathcal{F}_t^{i+1})_{0\leq t<t^\star(Z_i)}$-continuous-time martingale.
		\end{cor1}
		
		\begin{proof}
		The proof is similar to the one of Lemma 2.2 of \cite{Az12a}. Indeed, the conditional jump rate of $S_{i+1}$ given $Z_i,Z_{i+1}$ is $\widetilde{\lambda}(Z_i,Z_{i+1},\cdot)$ by Proposition \ref{txdesaut}.
		\end{proof}

\subsection{Ergodicity}
\label{sub:ergo}

In this part, we focus our attention on the asymptotic behavior of the Markov chains $(Z_n)_{n\geq0}$, $(Z_n,S_{n+1})_{n\geq0}$ and $(Z_n,Z_{n+1},S_{n+1})_{n\geq0}$. Our main objective is to state that one can apply the ergodic theorem to these Markov chains. First, we derive uniform lower bounds for the transition kernel $R$ and $G\widetilde{Q}(x,t^\star(x),y)$.

		\begin{rem}	\label{minR}
		Let $x,y\in E$ and $B\in\mathcal{B}(E)$. By $(\ref{pdmp:minoG})$ and $(\ref{pdmp:minoGQ})$, we have
		\begin{equation}
		G\widetilde{Q}(x,t^\star(x),y)    \geq   m_2  ,\label{pdmp:minoGQunif}
		\end{equation}
		where $m_2$ is given by
		\begin{equation*}
		m_2 = m \exp\left(-\int_0^{\|t^\star\|_{\infty}} M(s)\ud s\right) >0 .
		\end{equation*}
		In particular, $m_2$ do not depend on $x$ and $y$. Together with $(\ref{pdmp:minoR})$, we have
		\begin{equation}						\label{pdmp:minoRunif}
		R(x,B)\geq m_2 \mu(B) .
		\end{equation}
		\end{rem}
		
		\noindent
		According to the previous remark, one may state that the Markov chain $(Z_n)_{n\geq0}$ is ergodic.

		\begin{prop1}We have the following statements:
		\vspace{-0.3cm}
		\begin{enumerate}[a)]
		\item $(Z_n)_{n\geq0}$ is $\mu$-irreducible, aperiodic and admits a unique invariant measure, which we denote by $\nu$.
		\item There exist $\rho>1$ and $r>0$ such that,
		\begin{equation}															\label{CVgeom}
		\forall n\geq0,~\sup_{\xi\in E} \big\| R^n(\xi,\cdot) - \nu\big\|_{TV} \leq r \rho^{-n} ,
		\end{equation}
		where $\|\cdot\|_{TV}$ denotes the total variation norm (see \cite{HLL} for the definition).
		\item The Markov chain $(Z_n)_{n\geq0}$ is positive Harris-recurrent.
		\end{enumerate}
		\end{prop1}
		\begin{proof}
		By definition and Remark \ref{minR}, $(Z_n)_{n\geq0}$ is $\mu$-irre\-duci\-ble and aperiodic. In addition, the transition kernel $R$ obviously satisfies Doeblin's condition (see \cite[page 396]{MandT} for instance),
		$$\mu(B)>\epsilon ~\Rightarrow R(\xi,B)>m_2~\!\epsilon ,$$
		by $(\ref{pdmp:minoRunif})$. On the strength of Theorem 16.0.2 of \cite{MandT}, $(Z_n)_{n\geq0}$ admits a unique invariant measure $\nu$ since it is aperiodic and $(\ref{CVgeom})$ holds. In addition, from Theorem 4.3.3 of \cite{HLL}, $(Z_n)_{n\geq0}$ is positive Harris-recurrent.
		\end{proof}

\noindent
Now, we shall see that the sets with non-empty interior are charged by the invariant measure $\nu$.
		
		\begin{rem}											\label{pdmp:rem:nuA}
		The transition kernel $R$ admits a density according to the measure $\mu$. As a consequence, the invariant measure $\nu$ and the auxiliary measure $\mu$ are equivalent in the light of Theorem 10.4.9 of \cite{MandT}. This ensures that for any measurable set $A$ with non-empty interior, $\nu(A)>0$ by Assumptions \ref{hyps1}.					
		\end{rem}

\noindent
The following lemma deals with the limits of the sequences $(\widetilde{\nu}_n)_{n\geq0}$ and $(\eta_n)_{n\geq0}$.

		\begin{lem1}	\label{triomino}				
		For any initial distribution $\nu_0=\delta_{\{x\}}$, $x\in E$,
		$$\lim_{n\to+\infty} \big\|\widetilde{\nu}_n-\widetilde{\nu}\big\|_{TV}=0\qquad\text{and}\qquad \lim_{n\to+\infty} \big\|\eta_n - \eta\big\|_{TV}=0,$$
		where the limit distributions $\widetilde{\nu}$ and $\eta$ are given by,
		\begin{alignat}{2}
		\forall A\times B \in \mathcal{B}(E)^{\otimes 2},&~\widetilde{\nu}(A\times B) &&=~\int_A \nu(\ud x) R(x,B) ,	\label{pdmp:tildenulim} \\
		\forall A\times B\times\Gamma\in\mathcal{B}(E)^{\otimes 2}\otimes\mathcal{B}(\mathbf{R}_+),&~\eta(A\times B\times\Gamma) &&=~\int_{A\times B\times\Gamma} \gamma_{x,y}(\ud s) \widetilde{\nu}(\ud x\times\ud y) .		\label{pdmp:etalim}
		\end{alignat}
		\end{lem1}
		
		\begin{proof}
		Let $g$ be a measurable function bounded by $1$. By virtue of Fubini's theorem, we have
		$$
		\left| \int_{E\times E} g(x,y) \big(  \widetilde{\nu}_n (\ud x\times\ud y) - \widetilde{\nu}(\ud x\times \ud y) \big) \right| = \left| \int_{E} \big(\nu_n(\ud x) - \nu(\ud x)\big) \int_{E} g(x,y) R(x,\ud y)\right| ,
		$$
		from the expression of $\widetilde{\nu}_n$ $(\ref{pdmp:tildenu})$ and the definition of $\widetilde{\nu}$ $(\ref{pdmp:tildenulim})$. Thus,
		$$
		\left| \int_{E\times E} g(x,y)  \big(  \widetilde{\nu}_n (\ud x\times\ud y) - \widetilde{\nu}(\ud x\times \ud y) \big) \right| = \left|\int_E h(x) \big(\nu_n(\ud x) - \nu(\ud x)\big) \right| ,
		$$
		where $h:x\mapsto\int_E g(x,y)R(x,\ud y)$ is bounded by $1$ because $g$ is bounded by $1$ and $R$ is a transition kernel. Finally,
		$$ \big\|\widetilde{\nu}_n - \widetilde{\nu}\big\|_{TV}  \leq \big\|\nu_n-\nu\|_{TV} .$$
		One obtains the expected limit from $(\ref{CVgeom})$. Now, we state the second limit. Let $g'$ a measurable function bounded by $1$. In the light of Remark \ref{pdmp:rem:desint},
		$$\left|\int_{E\times E\times\mathbf{R}_+} g'(x,y,s) \big(  \eta_n(\ud x\times\ud y\times \ud s) - \eta(\ud x\times \ud y\times
\ud s) \big)\right| \leq \|h\|_{\infty} \|\widetilde{\nu}_n - \widetilde{\nu}\|_{TV},$$	
		by virtue of Fubini's theorem and with the function $h$ given by,
		$$\forall x,y\in E,~h(x,y)  =  \int_{\mathbf{R}_+} g'(x,y,s) \gamma_{x,y}(\ud s) .$$
		As $h$ is bounded by $1$, we have
		$$ \left| \int_{E\times E\times\mathbf{R}_+} g'(x,y,s) \big(  \eta_n(\ud x\times\ud y\times \ud s) - \eta(\ud x\times \ud y\times
\ud s) \big) \right| \leq \|\widetilde{\nu}_n - \widetilde{\nu}\|_{TV} .$$			 
		We previously established that $\|\widetilde{\nu}_n-\widetilde{\nu}\|_{TV}$ tends to $0$, thus $\|\eta_n - \eta\|_{TV}$ tends to $0$ too.
		\end{proof}

		\noindent
		The previous lemma induces the following result.
		
		\begin{prop1}We have the following statements:
		\vspace{-0.3cm}										
		\begin{enumerate}[a)]
		\item $(Z_n,Z_{n+1})_{n\geq0}$ (resp. $(Z_n,Z_{n+1},S_{n+1})_{n\geq0}$) is $\widetilde{\nu}$-irreducible (resp. $
\eta$-irreducible).
		\item $(Z_n,Z_{n+1})_{n\geq0}$ and $(Z_n,Z_{n+1},S_{n+1})_{n\geq0}$ are positive Harris-recur\-rent and ape\-rio\-dic Mar\-kov chains.
		\item $\widetilde{\nu}$ (resp. $\eta$) is the unique invariant measure of the chain $(Z_n,Z_{n+1})_{n\geq0}$ (resp. of the chain $(Z_n,Z_{n+1},S_{n+1})_{n\geq0}$).
		\end{enumerate}	
		\end{prop1}
		\begin{proof}
		This result is a consequence of Lemma \ref{triomino}. The proof is similar to the one of Proposition 4.2 given in \cite{Az12a}.
		\end{proof}

		\noindent
		According to the previous discussion, the Markov chains $(Z_n)_{n\geq0}$, $(Z_n,Z_{n+1})_{n\geq0}$ and $(Z_n,Z_{n+1},S_{n+1})_{n\geq0}$ are positive Harris-recurrent. As a consequence, one may apply the ergodic theorem to these Markov chains (Theorem 17.1.7 of \cite{MandT}).

		\section{Proofs of the main results}
\label{sec:proofs}

This section is dedicated to the presentation of most of the proofs of our main results provided in Section \ref{sec:problem}. In all this part, we consider a set $A\in\mathcal{C}_2$ and $e>0$. In particular, this ensures the existence of a family $\{B_1,\dots,B_p\}$, $p\geq2$, depending on $e$ and satisfying the last condition of Assumptions \ref{hyps1}.

\subsection{Estimation of $\widetilde{l}$}
	\label{sub:ltilde}

Here, we focus our attention on the estimation of the function $\widetilde{l}$, which is an approximation of the jump rate $\widetilde{\lambda}$. The precise definition of $\widetilde{l}$ may be found in Lemma \ref{defltilde}. The estimation of $\widetilde{l}$ is a keystone in our procedure for estimating the conditional density $f$. Here, our main objective is the presentation of the proof of Proposition \ref{estimationltilde}, which states at the end of this subsection. Two technical lemmas about lower and upper bounds are now presented. These results will be useful in all the sequel.
		
		\begin{lem1}
		Let $1\leq k\leq p$. For any $0\leq t<t^\star(A)$,
		\begin{equation} 	\label{pdmp:minoGtilde}
		 \inf_{x\in A,y\in B_k} \widetilde{G}(x,y,t)>0 .
		 \end{equation}
		\end{lem1}

		\begin{proof}
		We have from $(\ref{defG})$ and $(\ref{pdmp:minoGQunif})$, for any $x,y\in E$, $0\leq t<t^\star(x)$,
		\begin{eqnarray*}
		\widetilde{G}(x,y,t) &\geq& \frac{G\widetilde{Q}(x,t^\star(x),y)}{\int_0^{t^\star(x)} f\widetilde{Q}(x,s,y)\ud s ~\!+~\! G\widetilde{Q}(x,t^\star(x),y)} \\
		~&\geq&  \frac{m_2}{\big(\|t^\star\|_{\infty}~\! \|f\|_{\infty}   + 1\big) \|\widetilde{Q}\|_{\infty}},
		\end{eqnarray*}
		because $f$, $\widetilde{Q}$ and $t^\star$ are bounded (according to Assumptions \ref{hyps2}). This achieves the proof.
		\end{proof}

		\noindent
		One may also prove that the jump rate $\widetilde{\lambda}$ is a bounded function.
		\begin{lem1}	\label{lambdabounded}
		Let $x,y\in E$ and $0\leq t\leq t^\star(x)$. Thus,
		$$ \big|\widetilde{\lambda}(x,y,t)\big| \leq \frac{\|f\|_{\infty}\|\widetilde{Q}\|_\infty}{m_2} .$$
		\end{lem1}

		\begin{proof}
		By $(\ref{defLAMBDA})$, we have
		\begin{eqnarray*}
		\widetilde{\lambda}(x,y,t) &=&\frac{f\widetilde{Q}(x,t,y)}{\int_t^{t^\star(x)} f\widetilde{Q}(x,s,y)\ud s + G\widetilde{Q}(x,t^\star(x),y)} \\
		~&\leq& \frac{f\widetilde{Q}(x,t,y)}{G\widetilde{Q}(x,t^\star(x),y)} .
		\end{eqnarray*}
		Therefore, with $(\ref{pdmp:minoGQunif})$ and since $f$ and $\widetilde{Q}$ are bounded, one immediately obtains the expected result.
		\end{proof}


		\begin{rem} \label{pdmp:rem:nutildeAB}
		For any $1\leq k\leq p$, $\widetilde{\nu}(A\times B_k)>0$. Indeed,
		$$ \widetilde{\nu}(A\times B_k) = \int_A R(x,B_k) \nu(\ud x) .$$
		For any $x$, $R(x,B_k)\geq m_2\mu(B_k)$ by $(\ref{pdmp:minoRunif})$, thus,
		$$\widetilde{\nu}(A\times B_k) \geq m_2 \mu(B_k) \nu(A) .$$
		One may conclude because $\nu(A)>0$ by Remark \ref{pdmp:rem:nuA} and $\mu(B_k)>0$ since $B_k$ is a set with non-empty interior (see Assumptions \ref{hyps1}).
		\end{rem}
		
		\noindent
		The following results deal with the asymptotic properties of $Y_n$. Recall that the definition of the process $Y_n$ is given by $(\ref{yn})$. Its generalized inverse $Y_n^+$ is defined by $(\ref{pdmp:def:ynplus})$.

		\begin{lem1} \label{lem:inna}
		Let $1\leq k\leq p$. For any $x\in E$,
		$$ \forall 0\leq t<t^\star(A),~\frac{Y_n(A,B_k,t)}{n} \to \int_{A\times B_k} \widetilde{G}(x,y,t)\widetilde{\nu}(\ud x\times\ud y) ~\textit{$\mathbf{P}_x$-a.s.},$$
		when $n$ goes to infinity. In addition, this limit is strictly positive.
		\end{lem1}

		\begin{proof}
		By virtue of the ergodic theorem applied to the chain $(Z_n,Z_{n+1},S_{n+1})_{n\geq0}$,
		$$ \frac{1}{n} Y_n(A,B_k,t) \to \eta\big(A\times B_k\times [t,+\infty[ \big)~\textit{$\mathbf{P}_x$-a.s.}$$
		Together with $(\ref{pdmp:etalim})$ and Remark \ref{pdmp:rem:desint}, we have for any $0\leq t<t^\star(A)$,
		$$\frac{Y_n(A,B_k,t)}{n} \to \int_{A\times B_k} \widetilde{G}(x,y,t) \widetilde{\nu}(\ud x\times\ud y) ~\textit{$\mathbf{P}_x$-a.s.,}$$
		when $n$ goes to infinity. Since $\inf_{x\in A,y\in B_k}\widetilde{G}(x,y,t)>0$ by $(\ref{pdmp:minoGtilde})$ and $\widetilde{\nu}(A\times B_k)>0$ by Remark \ref{pdmp:rem:nutildeAB}, the limit is strictly positive.
		\end{proof}


		\begin{lem1}	\label{inna2}
		Let $1\leq k\leq p$, $0\leq t<t^\star(A)$ and $x\in E$. Then, for any integer $n$,
		\begin{equation*}	
		Y_n(A,B_k,t)^+\leq 1~\textit{$\prob_x$-a.s.}
		\end{equation*}
		and, as $n$ goes to infinity,
		\begin{eqnarray*}
		Y_n(A,B_k,t)^+ 								&\longrightarrow&	0\quad \prob_x\textit{-a.s.},		\\
		\mathbf{1}_{\{Y_n(A,B_k,t)=0\}} 				&\longrightarrow & 	0\quad\prob_x\textit{-a.s.},		\\
		\int_0^t \mathbf{1}_{\{Y_{n}(A ,B_k,s) = 0  \}}  \ud s 	&\longrightarrow&	0\quad\prob_x\textit{-a.s.}		
		\end{eqnarray*}
		\end{lem1}

		\begin{proof}
		This result is a corollary of Lemma \ref{lem:inna}. One may find a similar proof in \cite{Az12a}, Lemma 4.11.
		\end{proof}

		\noindent
		The function $\widetilde{l}$ is defined in the following proposition.
		\begin{prop1}		\label{defltilde}			
		Let $1\leq k\leq p$, $0\leq t<t^\star(A)$ and $x\in E$. When $n$ goes to infinity,
		$$Y_n(A,B_k,t)^+ \sum_{i=0}^{n-1} \widetilde{\lambda}(Z_i,Z_{i+1},t) \mathbf{1}_{\{Z_i\in A\}} \mathbf{1}_{\{Z_{i+1}\in B_k\}} \mathbf{1}_{\{S_{i
+1}\geq t\}} \longrightarrow  \widetilde{l}(A,B_k,t) ~\textit{$\prob_x$-a.s.,} $$
		where
		\begin{equation}	\label{pdmp:def:ltilde}
		\widetilde{l}(A,B_k,t) =  \frac{ \int_{A\times B_k} \widetilde{\lambda}(u,v,t) \widetilde{G}(u,v,t) \widetilde{\nu}(\ud u\times\ud v)}{\int_{A\times 
B_k} \widetilde{G}(u,v,t) \widetilde{\nu}(\ud u\times \ud v)} .
		\end{equation}
		In addition, $\widetilde{l}(A,B_k,\cdot)$ is continuous on $[0,t^\star(A)[$.
		\end{prop1}
		\begin{proof}
This is an application of the ergodic theorem to the chain $(Z_n,Z_{n+1},S_{n+1})_{n\geq0}$. One may refer the reader to the proof of Proposition 4.12 of \cite{Az12a}, which is similar. $\widetilde{l}(A,B_k,\cdot)$ is continuous because $\widetilde{\lambda}$ and $\widetilde{G}$ are continuous and bounded. Therefore, one may apply the theorem of continuity under the integral sign.
		\end{proof}

		Now, one may state some results about continuous-time martingales. First, we give a conditional independence property. Denote by $\mathcal{G}_n$ the $\sigma$-field $\sigma(Z_0,\dots,Z_n)$ for each integer $n$.

		\begin{prop1}
		\label{mozart2}
		Let $n$ be an integer and $1\leq i\leq n$. For each $j\neq i$, let $t_j\geq0$ and $t$ be a positive real number. Then, we have
		$$ \bigvee_{j\neq i} \mathcal{F}_{t_j}^{j}  ~   \underset{\mathcal{G}_n}{\bot} ~ \mathcal{F}_t^i 	\qquad\text{and}\qquad    
\mathcal{F}_t^i ~ \underset{\sigma(Z_{i-1},Z_i)}{\bot} ~ \mathcal{G}_n .$$
		We deduce from Proposition 6.8 of \cite{Kal} this direct corollary. For any $s<t$, $0\leq i\leq n-1$,
		$$
		\bigvee_{j\neq i+1} \mathcal{F}_{s}^j  ~   \underset{\mathcal{G}_n \vee \mathcal{F}_s^{i+1}}{\bot} ~ \mathcal{F}_t^{i+1}
		\qquad\text{and}\qquad
		\mathcal{F}_t^{i+1} ~\underset{\sigma(Z_i,Z_{i+1}) \vee \mathcal{F}_s^{i+1}}{\bot}~\mathcal{G}_n .
		$$
		\end{prop1}
		\begin{proof}
		The technical proof is deferred in Appendix \ref{appendixA}.
		\end{proof}

		\begin{theo}	\label{pdmp:theo:Mn}
		Let $1\leq k\leq p$. The process $M_n(A,B_k,\cdot)$, defined for any $0\leq t<t^\star(A)$, by
		$$M_n(A,B_k,t) = \sum_{i=0}^{n-1} M^{i+1}(t) \mathbf{1}_{\{Z_i\in A\}}\mathbf{1}_{\{Z_{i+1}\in B_k\}} ,$$
		is a $(\mathcal{G}_n\vee\bigvee_{i=0}^{n-1}\mathcal{F}_t^{i+1})_{0\leq t<t^\star(A)}$-continuous-time martingale.
		\end{theo}
		\begin{proof}
		One may recall that the process $M^{i+1}$ has already been defined by $(\ref{pdmp:def:Miplus1})$. The result is a corollary of Proposition \ref{mozart2} and Corollary \ref{cumparsita}. The arguments are given in the proof of Theorem 4.13 of \cite{Az12a}
		\end{proof}

		\noindent
		One may derive the expression of a new continuous-time martingale.

		\begin{lem1}		\label{crochettilde}																				
		Let $1\leq k\leq p$ and $n$ be an integer. The process $\widetilde{M}_n(A,B_k,\cdot)$, given for any $0\leq t<t^\star(A)$, by
		$$ \widetilde{M}_n(A,B_k,t) = \int_0^t Y_n(A,B_k,s)^+ \ud M_n(A,B_k,s) ,$$
		is a continuous-time martingale, whose predictable variation process $<\widetilde{M}_n(A,B_k)>$ satisfies, for any $x\in E$,
		$$\forall 0\leq t<t^\star(A),~<\widetilde{M}_n(A,B_k)>(t) \to 0~\prob_x\textit{-a.s.}~\text{when $n\to+\infty$.}$$
		\end{lem1}
		\begin{proof}
		This is a consequence of Lemma \ref{lambdabounded}, Lemma \ref{inna2} and Theorem \ref{pdmp:theo:Mn}. A similar proof may be found in \cite{Az12a}, Lemma 4.14.
		\end{proof}
	
		\begin{rem}
		This immediately induces that for any $0\leq t<t^\star(A)$,
		$$ \sup_{0\leq s\leq t}\big|\widetilde{M}_n(A,B_k,s)\big|\stackrel{\prob_x~}{\longrightarrow} 0 ,$$
		for any $x\in E$, by virtue of Lenglart's inequality. A reference may be found in the book \cite{And}, II.5.2.1. Lenglart's inequality.
		\end{rem}

		In the sequel, we are interested in the estimation of the function $\widetilde{l}(A,B_k,\cdot)$, defined for any $1\leq k\leq p$ by $(\ref{pdmp:def:ltilde})$. First, we shall estimate the function $\widetilde{L}(A,B_k,\cdot)$ defined by,
		$$ \forall 0\leq t<t^\star(A),~\widetilde{L}(A,B_k,t) = \int_0^t \widetilde{l}(A,B,s)\ud s.$$
		We consider the Nelson-Aalen type estimator $\widehat{\widetilde{L}}_n(A,B_k,\cdot)$ given by,
		$$ \forall 0\leq t<t^\star(A),~ \widehat{\widetilde{L}}_n(A,B_k,t) = \int_0^t Y_n(A,B_k,s)^+ \ud N_n(A,B_k,s) ,$$
		where $N_n(A,B_k,\cdot)$ is the counting process given by $(\ref{Nn})$. According to foregoing, the present framework is exactly the same one as in \cite{Az12a}.
		
		\begin{prop1}																									
		Let $1\leq k\leq p$, $0<t<t^\star(A)$ and $x\in E$. Then,
		$$ \sup_{0\leq s\leq t} \left|\widehat{\widetilde{L}}_n(A,B_k,s) -  \widetilde{L}(A,B_k,s)\right| \stackrel{\prob_x~}{\longrightarrow} 0~\text{when $n\to
+\infty$}.$$
		\end{prop1}

		\begin{proof}
		This is a corollary of Lemma \ref{crochettilde} and Proposition \ref{defltilde}. One may refer the interested reader to the proofs of Proposition 4.16 and Theorem 4.17 of \cite{Az12a}, which use similar arguments.
		\end{proof}

		\noindent
		We focus on smoothing this estimator in order to provide an estimate of $\widetilde{l}(A,B_k,\cdot)$. Recall that we consider a continuous kernel $K$ whose support is $[-1,1]$, $b>0$ and $0<t<t^\star(A)$. One may define our estimator of $\widetilde{l}(A,B_k,\cdot)$ by,
		$$\forall 0\leq s\leq t,~\widehat{\widetilde{l}}_{n,b,t} (A,B_k,s) = \frac{1}{b} \int_0^{t} K\Big(\frac{s-u}{b}\Big)\ud \widehat{\widetilde{L}}_n(A,B_k,u) .$$
		A direct calculus leads to the equivalent formula previously provided by $(\ref{estimateurltilde})$. In this context, one may prove Proposition \ref{estimationltilde}.

		\noindent
		\textbf{Proof of Proposition \ref{estimationltilde}.} The proof relies on the arguments provided in the one of Proposition 4.23 in \cite{Az12a}. \fin

\subsection{Approximation of $f$}
	\label{sub:approxf}
		In this subsection, we are interested in the proof of Proposition \ref{approxf}. For the sake of clarity, we denote for any $x,y\in E$ and $0\leq t\leq t^\star(x)$,
		\begin{equation}
		\label{pdmp:def:H}
		H(x,y,t) = \int_t^{t^\star(x)} f\widetilde{Q}(x,s,y)\ud s + G\widetilde{Q}(x,t^\star(x),y) .
		\end{equation}
		Therefore, by $(\ref{defLAMBDA})$, we have
		$$ H(x,y,t) \widetilde{\lambda}(x,y,t) = f\widetilde{Q}(x,t,y) .$$
		We shall establish some properties of $H$.

		\begin{lem1} \label{pdmp:lem:Hbounded}
		We have the following statements:
		\vspace{-0.3cm}
		\begin{enumerate}[a)]
		\item $\|H\|_{\infty}<+\infty$.
		\item For any $x,y\in E$, $0\leq t\leq t^\star(x)$, $H(x,y,t)\geq m_2$.
		\item Furthermore, there exists a constant $[H]_{Lip}>0$ such that,
		for any $x,y,u,v\in E$ and $0\leq t\leq t^\star(x)\wedge t^\star(u)$,
		\begin{equation} \label{Hlip}
		\big| H(x,y,t) - H(u,v,t)\big| \leq [H]_{Lip} d_2\big( (x,y),(u,v)\big) .
		\end{equation}
		\end{enumerate}
		\end{lem1}

		\begin{proof}
		First, by $(\ref{pdmp:def:H})$, we have
		$$ H(x,y,t) \leq \int_0^{t^\star(x)} f\widetilde{Q}(x,s,y)\ud s + G\widetilde{Q}(x,t^\star(x),y) .$$
		$G$ is bounded by $1$. In addition, $f$, $t^\star$ and $\widetilde{Q}$ are bounded functions according to Assumptions \ref{hyps1} and \ref{hyps2}. Thus,
		$$ H(x,y,t) \leq \|\widetilde{Q}\|_{\infty} \big( \|t^\star\|_{\infty} \|f\|_{\infty} +1\big) .$$
		Now, we prove that $m_2$ is a lower bound of $H$. Indeed, by $(\ref{pdmp:def:H})$ again, for any $x,y\in E$ and $0\leq t\leq t^\star(x)$,
		\begin{equation} \label{pdmp:eq:minoH}
		H(x,y,t)~\geq~G\widetilde{Q}(x,t^\star(x),y)~ \geq ~ m_2 .
		\end{equation}
		Finally, we state that $H$ is Lipschitz. By $(\ref{pdmp:def:H})$ again and by the triangle inequality,
		\begin{align*}
		\big|&H(x,y,t) - H(u,v,t)\big| \\
		&\leq \int_t^{t^\star(x)\vee t^\star(u)} \big|f\widetilde{Q}(u,s,v)-f\widetilde{Q}(x,s,y)\big|\ud s + \big|G\widetilde{Q}(u,t^\star(u),v)-G\widetilde{Q}(x,t^\star(x),y)\big| .
		\end{align*}
		Thus, since the product $f\widetilde{Q}$ is Lipschitz by $(\ref{fQlip})$, we have
		\begin{align*}
		\big| H(x,y,t) - H&(u,v,t)\big| \\
		&\leq \|t^\star\|_{\infty} [f\widetilde{Q}]_{Lip} d_2\big( (x,y),(u,v)\big) + \big|G\widetilde{Q}(u,t^\star(u),v)-G\widetilde{Q}(x,t^\star(x),y)\big| .
		\end{align*}
		Together with $(\ref{GQlip})$, we obtain $(\ref{Hlip})$ with
		$$[H]_{Lip} = \|t^\star\|_{\infty} [f\widetilde{Q}]_{Lip}+[G\widetilde{Q}]_{Lip} ,$$
		showing the three statements.
		\end{proof}

		\noindent
		Now, one may state that $\widetilde{\lambda}$ is also Lipschitz.
	
		\begin{lem1}		\label{pdmp:lem:lambdalip}
		There exists a constant $[\widetilde{\lambda}]_{Lip}>0$ such that, for any $x,y,u,v\in E$ and $0\leq t<t^\star(x)\wedge t^\star(u)$,
		$$
		\Big|\widetilde{\lambda}(x,y,t)-\widetilde{\lambda}(u,v,t)\Big| \leq [\widetilde{\lambda}]_{Lip} d_2\big((x,y),(u,v)\big) .
		$$
		\end{lem1}
		
\begin{proof}
From $(\ref{defLAMBDA})$ and $(\ref{pdmp:def:H})$, we have
$$
\Big|\widetilde{\lambda}(x,y,t) - \widetilde{\lambda}(u,v,t)\Big| = \frac{ \big|f\widetilde{Q}(x,t,y)H(u,v,t)-f\widetilde{Q}(u,t,v)H(x,y,t) \big|}{ H(u,v,t) H(x,y,t)}.
$$
Therefore, together with $(\ref{pdmp:eq:minoH})$, we obtain
$$
\Big|\widetilde{\lambda}(x,y,t) - \widetilde{\lambda}(u,v,t)\Big| \leq  \frac{1}{m_2^2}\big|f\widetilde{Q}(x,t,y)H(u,v,t)-f\widetilde{Q}(u,t,v)H(x,y,t) \big| .
$$
By the triangle inequality, this induces that
\begin{align*}
\big| \widetilde{\lambda}&(x,y,t) - \widetilde{\lambda}(u,v,t) \big| \\
&\leq \frac{1}{m_2^2}\Big( \big| f\widetilde{Q}(x,t,y)H(u,v,t) - f\widetilde{Q}(x,t,y)H(x,y,t)\big| \\
& \qquad + ~ \big|f\widetilde{Q}(x,t,y)H(x,y,t) - f\widetilde{Q}(u,t,v)H(x,y,t) \big|\Big) \\
&\leq \frac{1}{m_2^2} \Big( \|f\|_\infty\|\widetilde{Q}\|_{\infty} \big|H(u,v,t)-H(x,y,t)\big| + \|H\|_{\infty} \big|f\widetilde{Q}(x,t,y)-f\widetilde{Q}(u,t,v)\big|\Big) .
\end{align*}
$f\widetilde{Q}$ is Lipschitz in the light of Assumptions \ref{hyps2} and $H$ is Lipschitz by Lemma \ref{pdmp:lem:Hbounded}. Hence,
$$\big|\widetilde{\lambda}(x,y,t)-\widetilde{\lambda}(u,v,t)\big| \leq [\widetilde{\lambda}]_{Lip} d_2\big((x,y),(u,v)\big) ,$$
where $[\widetilde{\lambda}]_{Lip}$ is given by
$$ [\widetilde{\lambda}]_{Lip} = \frac{\|f\|_\infty\|\widetilde{Q}\|_{\infty}[H]_{Lip}+\|H\|_{\infty} [f\widetilde{Q}]_{Lip}}{m_2^2}.$$
This achieves the proof.
\end{proof}

		\begin{rem}
		Under the last point of Assumptions \ref{hyps1}, one may state a new property of the measure $\mu$. By $(\ref{minQ})$, we have for any $1\leq k\leq p$,
		$$ \overline{Q}(x,t,B_k) = \int_{B_k} \widetilde{Q}(x,t,y)\mu(\ud y) \geq m \mu(B_k) .$$
		Summing on $k$ yields to
		\begin{equation} \sum_{k=1}^p \mu(B_k) \leq \frac{1}{m} ,	\label{majMU} .\end{equation}
		\end{rem}
		
		\noindent
		In the following lemma, we establish that $\widetilde{l}$ is an approximation of $\widetilde{\lambda}$.
		
		\begin{lem1} \label{pdmp:approxtilde}
		Let $1\leq k\leq p$, $x\in A$ and $y\in B_k$. Let $0\leq t<t^\star(A)$. Then,
		\begin{equation}	\label{pdmp:link:lambdal}
		\big|\widetilde{\lambda}(x,y,t) - \widetilde{l}(A,B_k,t)\big| \leq [\widetilde{\lambda}]_{Lip} \text{\normalfont diam}_2~\! A\times B_1.
		\end{equation}
		\end{lem1}
		\begin{proof}
		By $(\ref{pdmp:def:ltilde})$, we have
		\begin{eqnarray*}
		\big|\widetilde{\lambda}(x,y,t) - \widetilde{l}(A,B_k,t) \big| &\leq & \Bigg|  \widetilde{\lambda}(x,y,t) -   \frac{ \int_{A\times B_k} \widetilde{\lambda}(u,v,t) \widetilde{G}(u,v,t) \widetilde{\nu}(\ud u\times\ud v)}{\int_{A\times B_k} \widetilde{G}(u,v,t) \widetilde{\nu}(\ud u\times \ud v)}  \Bigg| \\
		~&\leq & \frac{ \int_{A\times B_k} \big| \widetilde{\lambda}(x,y,t) - \widetilde{\lambda}(u,v,t)\big| \widetilde{G}(u,v,t) \widetilde{\nu}(\ud u\times\ud v)}{\int_{A\times B_k} \widetilde{G}(u,v,t) \widetilde{\nu}(\ud u\times \ud v)}  \\
		~ & \leq& [\widetilde{\lambda}]_{Lip}  \text{\normalfont diam}_2~\!A\times B_k,
		\end{eqnarray*}
		by virtue of Lemma \ref{pdmp:lem:lambdalip}. This yields to the expected result, since $B_1$ is assumed to be the set with the largest diameter.
		\end{proof} 

		\noindent
		Now, one may establish Proposition \ref{approxf}. The function $\widetilde{H}(A,B_k,\cdot)$ appearing in this result is defined for any $0\leq t<t^\star(A)$, by
		\begin{equation}
		\label{eq:def:Htilde}
		\widetilde{H}(A,B_k,t) = \frac{1}{\nu(A)} \int_{A\times B_k} H(x,y,t) \mu(\ud y)\nu(\ud x) ,
		\end{equation}
		where the function $H$ has been defined by $(\ref{pdmp:def:H})$.

		\noindent
		\textbf{Proof of Proposition \ref{approxf}.} First, we show that
		\begin{align}
		\frac{1}{\nu(A)} \Bigg|\int_A f(x,t)\nu(\ud x) -  \sum_{k=1}^p \widetilde{l}(A,B_k,t) \int_{A\times B_k}& H(x,y,t) \mu(\ud y)\nu(\ud x) \Bigg|
		\nonumber \\
		&\leq \frac{1}{m} [\widetilde{\lambda}]_{Lip} \|H\|_{\infty} \text{\normalfont diam}_2~\! A\times B_1 . \label{pdmp:ing}
		\end{align}
		Let $1\leq k\leq p$, $x\in A$ and $y\in B_k$. Multiplying $(\ref{pdmp:link:lambdal})$ by $H(x,y,t)$, we obtain
		$$ \Big| f\widetilde{Q}(x,t,y) - \widetilde{l}(A,B_k,t) H(x,y,t) \Big| \leq [\widetilde{\lambda}]_{Lip} \|H\|_{\infty} \text{diam}_2~\! A\times B_1 ,$$
		because $H$ is bounded on the strength of Lemma \ref{pdmp:lem:Hbounded}. We integrate on $B_k$ according to $\mu(\ud y)$. We obtain		$$ \left|f\overline{Q}(x,t,B_k) - \widetilde{l}(A,B_k,t) \int_{B_k} H(x,y,t)\mu(\ud y)\right| \leq \mu(B_k) [\widetilde{\lambda}]_{Lip} \|H\|_{\infty} \text{diam}_2~\! A\times B_1 .$$
		Summing on $k$ between $1$ and $p$ yields to
		$$ \left| f(x,t) - \sum_{k=1}^p \widetilde{l}(A,B_k,t) \int_{B_k} H(x,y,t)\mu(\ud y)\right| \leq \frac{1}{m} [\widetilde{\lambda}]_{Lip} \|H\|_{\infty} \text{diam}_2~\! A\times B_1 ,$$
		with $(\ref{majMU})$. Finally, we integrate on $A$ according to $\nu(\ud x)$.
		\begin{align}
		\Bigg|\int_A f(x,t)\nu(\ud x) - \sum_{k=1}^p\widetilde{l}(A,B_k,t)\int_{A\times B_k} &H(x,y,t) \nu(\ud x)\mu(\ud y)\Bigg| \nonumber \\
		&\leq \frac{\nu(A)}{m} [\widetilde{\lambda}]_{Lip} \|H\|_{\infty} \text{\normalfont diam}_2~\! A\times B_1 . \label{pdmp:interm1}
		\end{align}
		Doing the ratio of each term of $(\ref{pdmp:interm1})$ by $\nu(A)$ leads to $(\ref{pdmp:ing})$. On the other hand, we have
		\begin{equation} \label{pdmp:interm2}
		\Bigg| f(\xi,t) - \frac{\int_A f(x,t)\nu(\ud x)}{\nu(A)} \Bigg| \leq [f]_{Lip} \text{diam}~\!A ,
		\end{equation}		
		because $f$ is Lipschitz by Assumptions \ref{hyps2}. Finally, $(\ref{pdmp:ing})$ and $(\ref{pdmp:interm2})$ give the proof by virtue of the triangle inequality. \fin

\subsection{Estimation of $\widetilde{H}$}
	\label{sub:Htilde}

		Let $\xi\in A$ and $0\leq t<t^\star(A)$. For estimating $f(\xi,t)$, we need to estimate $\widetilde{l}(A,B_k,t)$ and $\widetilde{H}(A,B_k,t)$ for each $1\leq k\leq p$, according to Proposition \ref{approxf}. We shall see that this quantity may be seen as a conditional probability.

		\begin{prop1} \label{pdmp:prop:interp}
		Let $0\leq t<t^\star(A)$. Then,
		$$\int_{A\times B_k} H(x,y,t)\nu(\ud x)\mu(\ud y) = \prob_{\nu} (S_1 > t,Z_1\in B_k,Z_0\in A) .$$
		This immediately induces that
		$$\widetilde{H}(A,B_k,t)  =   \prob_{\nu}(S_1>t,Z_1\in B_k|Z_0\in A) .$$
		One may recall that $\nu(A)$ is a strictly positive number according to Remark \ref{pdmp:rem:nuA}.
		\end{prop1}
		\begin{proof}
		By $(\ref{pdmp:etalim})$ and Remark \ref{pdmp:rem:desint}, we have for any $t<t^\star(A)$,
		\begin{eqnarray}
		\prob_{\nu}(S_1 > t,Z_1\in B_k,Z_0\in A)  &=& \eta\big(A\times B_k\times [t,+\infty[\big) \label{pdmp:eta_interm01} \\
		~&=& \int_{A\times B_k} \widetilde{G}(x,y,t) \widetilde{\nu}(\ud x\times\ud y)  . \nonumber
		\end{eqnarray}
		Hence,
		$$\prob_{\nu}(S_1 > t,Z_1\in B_k,Z_0\in A) = \int_{A\times B_k} \widetilde{G}(x,y,t) \nu(\ud x) R(x,\ud y) ,$$
		by $(\ref{pdmp:tildenulim})$. Thus, by the expression of $R$ $(\ref{pdmp:exprR})$, the definition of $\widetilde{G}$ $(\ref{defG})$ and the definition of $H$ $(\ref{pdmp:def:H})$, we have
		$$
		\widetilde{G}(x,y,t) R(x,\ud y) = H(x,y,t) \mu(\ud y).
		$$
		Therefore,
		\begin{equation}	\label{pdmp:eta_interm02}
		\prob_{\nu}(S_1 > t,Z_1\in B_k,Z_0\in A) = \int_{A\times B_k} H(x,y,t) \nu(\ud x) \mu(\ud y) ,
		\end{equation}
		showing the result.
		\end{proof}

		\noindent
		Let $A\in\mathcal{C}_2$, $1\leq k\leq p$ and $0\leq t<t^\star(A)$. One may estimate the conditional probability
		$$\widetilde{H}(A,B_k,t)   =   \prob_{\nu}(S_1>t,Z_1\in B_k|Z_0\in A) ,$$
		by its empirical version given by $(\ref{pdmp:def:estpn})$. The uniform convergence of this estimator lies in Proposition \ref{unifps}. First, we establish the pointwise convergence.

		\begin{lem1}	\label{simpleps}
		For any $0\leq t<t^\star(A)$, $x\in E$, we have as $n$ goes to infinity,
		$$ \left| \widehat{p}_n(A,B_k,t)  -    \widetilde{H}(A,B_k,t)   \right|   \to 0 ~  \prob_{x}\textit{-a.s.}$$
		\end{lem1}
		
	\begin{proof}
	One applies the ergodic theorem to $(Z_n,Z_{n+1},S_{n+1})_{n\geq0}$. Thus, when $n$ goes to infinity,
	$$
	\frac{1}{n}\sum_{i=0}^{n-1}\mathbf{1}_{\{Z_i\in A\}}\mathbf{1}_{\{Z_{i+1}\in B_k\}}\mathbf{1}_{\{S_{i+1}>t\}} \to \eta\big(A\times B_k\times]t,+\infty[\big)~\prob_x\textit{-a.s.}
	$$
	Together with $(\ref{pdmp:eta_interm01})$ and $(\ref{pdmp:eta_interm02})$,
	\begin{equation}
	\frac{1}{n}\sum_{i=0}^{n-1}\mathbf{1}_{\{Z_i\in A\}}\mathbf{1}_{\{Z_{i+1}\in B_k\}}\mathbf{1}_{\{S_{i+1}>t\}}\to\int_{A\times B_k} H(x,y,t)\nu(\ud x)\mu(\ud y)~\prob_x\textit{-a.s.} \label{pdmp:proof:ergo1}
	\end{equation}
	In addition, by applying the ergodic theorem to the Markov chain $(Z_n)_{n\geq0}$, we have
	\begin{equation}
	\frac{1}{n}\sum_{i=0}^{n-1}\mathbf{1}_{\{Z_i\in A\}} \to \nu(A)~\prob_x\textit{-a.s.} \label{pdmp:proof:ergo2}
	\end{equation}
	Combining $(\ref{pdmp:def:estpn})$, $(\ref{pdmp:proof:ergo1})$ and $(\ref{pdmp:proof:ergo2})$, we show the expected result.
	\end{proof}

	\noindent
	Now, we can give the proof of Proposition \ref{unifps}, which states that $\widehat{p}_n(A,B_k,\cdot)$ is a consistent estimator.

	\noindent
	\textbf{Proof of Proposition \ref{unifps}.} Since $A$ and $k$ are fixed in the sequel, we will write $\widetilde{H}(t)$ instead of $\widetilde{H}(A,B_k,t)$ for the sake of readability. From Lemma \ref{simpleps}, for any $s$, $\prob_{x}(\Upsilon_s) = 1$ where the set $\Upsilon_s$ is defined by
	$$\Upsilon_s = \big\{\widehat{p}_n(A,B_k,s) \to \widetilde{H}(s) \big\} .$$ 	
	First, we prove that $\widetilde{H}$ is a strictly decreasing function. By $(\ref{pdmp:def:H})$ and by virtue of the theorem of derivation under the integral sign, $\widetilde{H}'$ satisfies,
	\begin{eqnarray*}
	\forall 0\leq s\leq t,~ \widetilde{H}'(s) &=& - \frac{1}{\nu(A)} \int_A  f(\xi,s)\overline{Q}(\xi,s,B_k) \nu(\ud\xi) \\
	~&\leq& -  \frac{\mu(B_k) m}{\nu(A)} \int_A f(\xi,s) \nu(\ud \xi) ,
	\end{eqnarray*}
	by $(\ref{minQ})$. According to Assumptions \ref{hyps2}, $f$ is strictly positive. Thus, $\int_A f(\xi,s)\nu(\ud\xi)>0$ because $\nu(A)>0$. As a consequence, $\widetilde{H}'(s)<0$ for any $s$, and $\widetilde{H}$ is, therefore, strictly decreasing. In particular, this is a one-to-one correspondence mapping from $[0,t]$ into $[a,b]$, where
	$$[a,b] = \big[\widetilde{H}(t), \prob_{\nu}(Z_1\in B_k| Z_0 \in A)\big],$$
	since, in the light of Proposition \ref{pdmp:prop:interp},
	$$\widetilde{H}(0)=\prob_\nu(Z_1\in B_k| Z_0 \in A) .$$
	For any couple $(l,m)$ of integers, with $0\leq l\leq m$, let us consider
	$$X(l,m) = \widetilde{H}^{-1}\left( a+\frac{(m-l)(b-a)}{m} \right) .$$
	By construction, we have for any $0\leq l\leq m-1$,
	$$ \widetilde{H}( X(l+1,m)) - \widetilde{H}( X(l,m)) = \frac{b-a}{m} .$$
	Hence, for $s\in[X(l,m) , X(l+1,m)[$ with $0\leq l\leq m-2$, or for $s\in[X(m-1,m),t]$, we have
	\begin{eqnarray}
	\widehat{p}_n(A,B_k,s) - \widetilde{H}(s) 		&\geq&	 \widehat{p}_n( A,B_k, X(l+1,m)) - 	\widetilde{H}(X(l,m)) \nonumber\\
	~								&\geq&	 \widehat{p}_n( A,B_k, X(l+1,m)) - 	\widetilde{H}(X(l+1,m)) + \frac{b-a}{m} , \label{majbetapn} \\
	\widehat{p}_n(A,B_k,s) - \widetilde{H}(s) 		&\leq&	 \widehat{p}_n( A,B_k, X(l,m) ) -		\widetilde{H}(X(l+1,m)) \nonumber \\
	~								&\leq& 	\widehat{p}_n( A,B_k,  X(l,m) ) - 	\widetilde{H}(X(l,m)) 	- \frac{b-a}{m} , \label{minbetapn}
	\end{eqnarray}
	because $\widetilde{H}$ and $\widehat{p}_n(A,B_k,\cdot)$ are decreasing functions. Let $0\leq s\leq t$. There are two possibilities: either $s\in[X(m-1,m),t]$, or there exists $0\leq l\leq m-2$ such that $s\in[X(l,m) , X(l+1,m)[$. Together with $(\ref{majbetapn})$ and $(\ref{minbetapn})$,
	$$ \sup_{0\leq s\leq t} \Big| \widehat{p}_n(A,B_k,s) - \widetilde{H}(s) \Big| \leq \alpha_{n,m} + \frac{b-a}{m},$$
	where
	$$ \alpha_{n,m} = \sup_{0\leq l\leq m} \Big| \widehat{p}_n(A,B_k,  X(l,m) )   -   \widetilde{H}(X(l,m)) \Big| .$$
	Let
	$$\Omega_m = \bigcap_{l=0}^{m} \Upsilon_{X(l,m)} \quad\text{and}\quad \Omega_{\infty} = \bigcap_{m\geq 2} \Omega_m .$$
	Consequently, when $n$ goes to infinity
	$$\forall\omega\in\Omega_m,~\alpha_{n,m}(\omega)\to 0 .$$
	In addition, $\prob_x(\Omega_m)=1$ because $\Omega_m$ is a finite intersection of sets with probability one. Finally,
	$$\limsup_{n\to+\infty} \sup_{0\leq s\leq t} \Big|\widehat{p}_n(A,B_k,s)-\widetilde{H}(s)\Big| \leq \inf_{m\geq 2} \frac{b-a}{m} ~ \prob_
{x}\textit{-a.s.,}
	$$
	since $\prob_{x}(\Omega_\infty)=1$ as a countable intersection of sets with probability one. \fin

\section{Numerical example}
\label{s:simu}

In this section, we present a short simulation study for illustrating the convergence result stated in Theorem \ref{pdmp:theo:fin}. We consider a PDMP $(X_t)_{t\geq0}$ defined on the state space $\mathcal{D}\times\mathcal{I}$, where
$$\mathcal{D}=\{x\in\mathbf{R}^2~:~\|x\|_2\leq 1\}\qquad\text{and}\qquad\mathcal{I}=]0,2\pi[.$$
In addition, we assume that the process starts from $X_0=(0,0,\pi)$. Let us consider $x=(x_1,x_2)\in\mathcal{D}$ and $\theta\in\mathcal{I}$. The flow $\Phi$ satisfies,
$$\forall t\geq0,~ \Phi( (x,\theta) ,t ) = \big(x_1+t\cos(\theta) , x_2+t\sin(\theta) , \theta\big).$$
The process $(X_t)_{t\geq0}$ has two components: intuitively, the first one represents the location in $\mathcal{D}$, while the second one models the direction of the deterministic motion which takes place in $\mathcal{D}$. The jump rate $\lambda$ is given by $\lambda((x,\theta))=5+\|x\|_2$. Since $\lambda$ does not depend on $\theta$, we will write $\lambda(x)$. Finally, the transition kernel $Q$ is defined for any $A\in\mathcal{B}(\mathbf{R}^2)$ and $B\in\mathcal{B}(\mathbf{R})$, by
$$Q\big((x,\theta),A\times B\big)=\frac{1}{K_x} \int_A\mathbf{1}_{\mathcal{D}}(y)\exp\left(-\frac{1}{2\sigma^2}\|y-x\|_2^2\right)\ud y ~\int_B \mathbf{1}_{\mathcal{I}}(u)\ud u ,$$
with $K_x$ as the normalizing constant and $\sigma^2$ as a parameter of variance. When a jump occurs, the new location on $\mathcal{D}$ is chosen according to a gaussian distribution centered in the previous position. The new angle is randomly chosen in $\mathcal{I}$.

This process may model the movement of a bacteria in a closed environment (see for instance \cite{MAL}). The bacteria moves on a line with a constant speed. It spontaneously and randomly changes its direction. During the rotation, the location may be a little modified (the parameter $\sigma^2$ must be chosen small, $\sigma^2=10^{-4}$ in our simulation study). Next, the bacteria moves again on a line. A change of direction occurs also when the bacteria tries to run through the boundary of its environment. In our model, the jump rate $\lambda$ depends only on the distance between the bacteria and the origin.

In the sequel, we focus our attention on the estimation of the conditional density $f(x_0,t)$ for $x_0=(0,0)$. We give an explicit formula of $f(x_0,t)$,
$$\forall t\geq0,~f(x_0,t) = (5+t)\exp\big( -t (5+t/2)\big).$$
For any $y\in E$ and $t$, we choose to approximate the jump rate $\widetilde{\lambda}(x_0,y,t)$ by $\widetilde{l}(A,B_k,t)$, with $A=]-\varepsilon,\varepsilon[^2$, $\varepsilon=0.1$, and $B_1=A$, $B_2=\mathcal{D}\setminus A$, that is, $(B_k)$ is the simplest partition of $\mathcal{D}$ that one may consider. In this context, $t^\star(A)=0.9$. Therefore, we decide to estimate $\widetilde{l}(A,B_k,t)$ by integrating between $0$ and $0.8$, which is less than $t^\star(A)$. Finally, we provide an estimate of $f$ within the interval $[0.05,0.75]$, which is a proper subset of $[0,0.8]$. 

We simulate a long trajectory of the process: the observation of $50000$ jumps is available to estimate $f$. The number of visits in $A$ is $5330$. In addition, the chosen bandwith $\beta_n(A,B_k)$ can be written in the following way,
$$\beta_n(A,B_k) = \frac{1}{h_n(A,B_k)^\alpha},$$
where $h_n(A,B_k)$ denotes the random number of visits in $A$ followed by a visit in $B_k$, and $\alpha=1/3$. Figure \ref{lab:fig1} is given to illustrate the good behavior of the estimator of $f$.

	\begin{figure}[!h]
	\centering
	\includegraphics[width=0.9\textwidth]{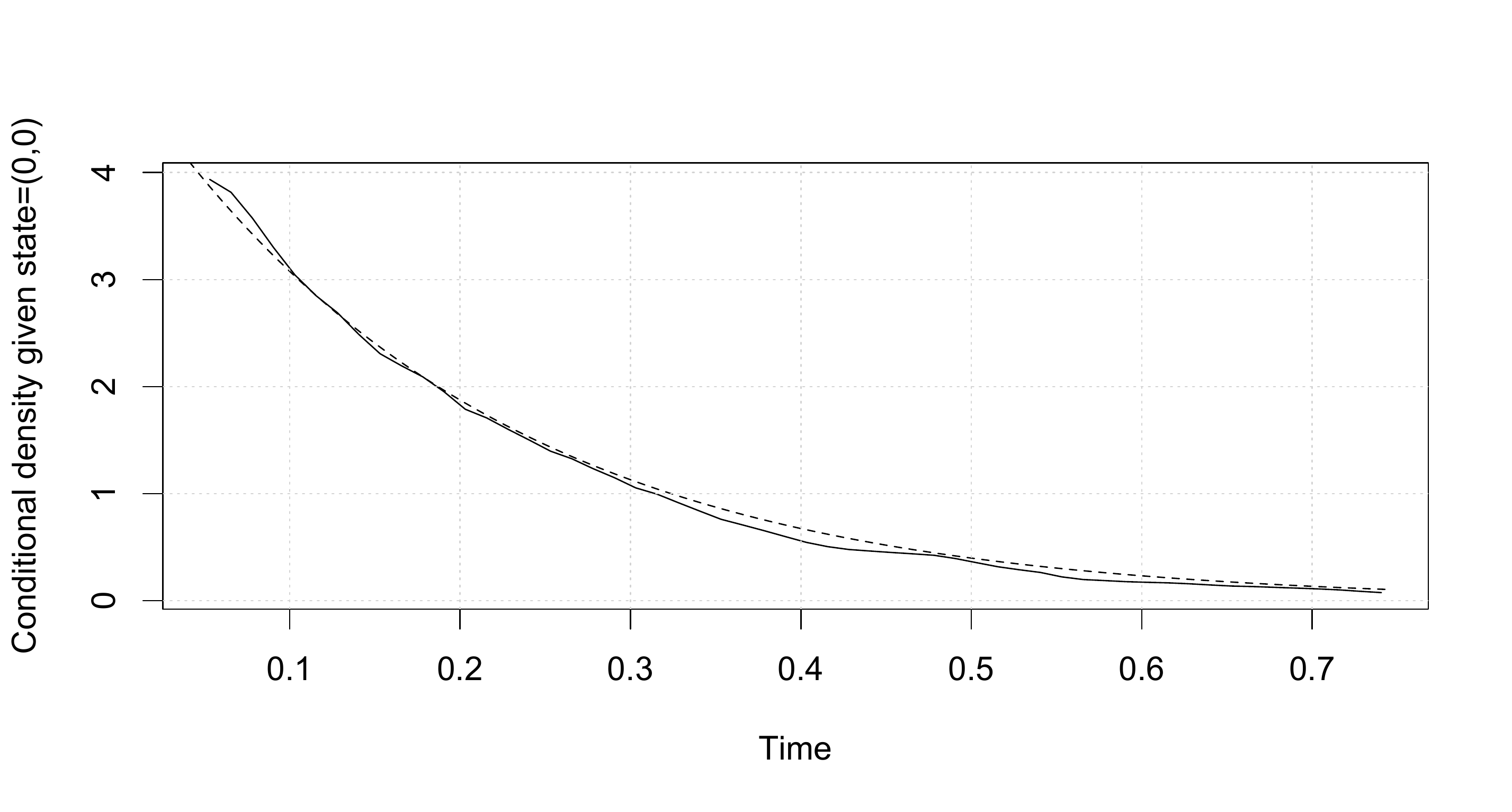}
	\caption{Estimation of the conditional density $f(x_0,t)$ with $x_0=(0,0)$ and $0.05\leq t\leq 0.75$. The estimate is drawn in solid line, while the exact density is in dashed line.}
	\label{lab:fig1}
	\end{figure}

\appendix
\section{Proof of Proposition \ref{mozart2}} \label{appendixA}

For the first conditional independence, let $h_1,\dots,h_n$ be some bounded measurable functions mapping from $\mathbf{R}_+$ into $\mathbf
{R}$. We have $\mathcal{G}_n \subset \mathcal{G}_n\vee\sigma(\delta_{n-1})$ and, by $(\ref{pdmp:dyn2})$, $h_n(S_n)$ is $\mathcal{G}_n\vee\sigma(\delta_{n-1})$-measurable. Thus,
\begin{align}
\esp_{\nu_0} & \big[ h_1(S_1)\dots h_n(S_n) |\mathcal{G}_n \big]  \nonumber \\
&= \esp_{\nu_0} \Big[ \esp_{\nu_0}\big[ h_1(S_1)\dots h_n(S_n) | \sigma(\delta_{n-1})\vee \mathcal{G}_n\big]   \big|    \widetilde{\mathcal
{G}}_n \Big] \nonumber \\
&= \esp_{\nu_0} \Big[   h_n(S_n)    \esp_{\nu_0}\big[ h_1(S_1)\dots h_{n-1}(S_{n-1}) | \sigma(\delta_{n-1})\vee\mathcal{G}_n\big]   \big|     
\mathcal{G}_n \Big]  . \label{pdmp:ind1}
\end{align}
Furthermore, always by $(\ref{pdmp:dyn2})$, $\sigma(\delta_{n-1})\vee\mathcal{G}_n \subset \sigma(\delta_{n-1})\vee\mathcal{G}_{n-1} \vee \sigma(\varepsilon_{n-1})$. Consequently,
\begin{align}
&\esp_{\nu_0}\big[ h_1(S_1)\dots h_{n-1}(S_{n-1}) | \sigma(\delta_{n-1})\vee\mathcal{G}_n\big] \nonumber\\
&=\esp_{\nu_0}\Big[  \esp_{\nu_0}\big[ h_1(S_1)\dots h_{n-1}(S_{n-1}) | \sigma(\delta_{n-1})\vee\mathcal{G}_{n-1} \vee \sigma(\varepsilon_{n-1}) 
\big]  \big|    \sigma(\delta_{n-1})\vee\mathcal{G}_n \Big] .	\label{pdmp:ind3}
\end{align}
Nevertheless, $h_1(S_1)\dots h_{n-1}(S_{n-1})$ is $\mathcal{G}_{n-1}\vee \sigma(\delta_0,\dots,\delta_{n-2})$-measurable, and moreover,
$$ \sigma(\varepsilon_{n-1})\vee\sigma(\delta_{n-1}) ~\bot ~ \mathcal{G}_{n-1}~\! \vee ~\! \sigma(\delta_0 , \dots, \delta_{n-2}) .$$
Together with $(3)$ \cite[page 308]{Chung},
$$\esp_{\nu_0}\big[ h_1(S_1)\dots h_{n-1}(S_{n-1}) | \sigma(\delta_{n-1})\vee\mathcal{G}_{n-1} \vee \sigma(\varepsilon_{n-1}) \big] = \esp_
{\nu_0}\big[ h_1(S_1)\dots h_{n-1}(S_{n-1}) | \mathcal{G}_{n-1} \big] .$$
Finally, with $(\ref{pdmp:ind3})$,
$$
\esp_{\nu_0}\big[ h_1(S_1)\dots h_{n-1}(S_{n-1}) | \sigma(\delta_{n-1})\vee\mathcal{G}_n\big]  = \esp_{\nu_0}\big[ h_1(S_1)\dots h_{n-1}(S_
{n-1}) | \mathcal{G}_{n-1} \big]  .
$$
Thus, by $(\ref{pdmp:ind1})$,
$$\esp_{\nu_0}  \big[ h_1(S_1)\dots h_n(S_n) | \mathcal{G}_n \big] =  \esp_{\nu_0}\big[ h_1(S_1)\dots h_{n-1}(S_{n-1}) | \widetilde{\mathcal
{G}}_{n-1} \big]   \esp_{\nu_0}[ h_n(S_n) | \mathcal{G}_n] .$$
Therefore, by a straightforward induction, we have
$$  \esp_{\nu_0}  \big[ h_1(S_1)\dots h_n(S_n) | \mathcal{G}_n \big]    =   \prod_{i=1}^n \esp_{\nu_0} \big[h_i(S_i) | \mathcal{G}_i\big] .$$
Taking for $j\neq i$, $h_j=\mathbf{1}$, leads to
\begin{equation} \label{indep04}
\esp_{\nu_0} \big[h_i(S_i) | \mathcal{G}_n\big] = \esp_{\nu_0} \big[h_i(S_i) | \mathcal{G}_i\big]  .
\end{equation}
Hence,
$$  \esp_{\nu_0}  \big[ h_1(S_1)\dots h_n(S_n) | \mathcal{G}_n \big]    =   \prod_{i=1}^n \esp_{\nu_0} \big[h_i(S_i) | \mathcal{G}_n
\big] .$$
This shows that
$$ \bigvee_{j\neq i} \sigma(S_j) ~ \underset{\mathcal{G}_n}{\bot} ~ \sigma(S_i) ,$$
and this directly induces the expected result. For the second conditional independence, let now $h_1:\mathbf{R}_+\to\mathbf{R}$ and $h_2:E^{n+1}\to\mathbf{R}$ be some bounded measurable functions. $\sigma(Z_{i-1},Z_i)\subset\mathcal{G}_n$, thus,
\begin{align}
\mathbf{E}_{\nu_0} & \big[ h_1(S_i) h_2(Z_0,\dots,Z_{n}) | \sigma(Z_{i-1},Z_i) \big] 	\nonumber\\
&= \mathbf{E}_{\nu_0} \Big[  \mathbf{E}_{\nu_0}\big[ h_1(S_i) h_2(Z_0,\dots,Z_{n})| \mathcal{G}_n \big] \big| \sigma(Z_{i-1},Z_i) \Big]  \nonumber \\
&= \mathbf{E}_{\nu_0} \Big[ h_2(Z_0,\dots,Z_{n})  \mathbf{E}_{\nu_0} \big[h_1(S_i)    	| \mathcal{G}_n	\big]   \big| \sigma(Z_{i-1},Z_i) 
\Big] . \label{pdmp:ind2}
\end{align}
We shall prove that
\begin{equation}	\label{indep05}
\mathbf{E}_{\nu_0} \big[h_1(S_i)    	| \mathcal{G}_n	\big] = \mathbf{E}_{\nu_0} \big[h_1(S_i)    	|\sigma(Z_{i-1},Z_i)\big] .
\end{equation}
By $(\ref{indep04})$, we have
$$ \mathbf{E}_{\nu_0} \big[h_1(S_i)    	| \mathcal{G}_n	\big] = \mathbf{E}_{\nu_0} \big[h_1(S_i)    	|\sigma(Z_0,\dots,Z_i)\big] .$$
Therefore, in order to state $(\ref{indep05})$, we have to prove that
\begin{equation} \label{indep06}
\sigma(S_i) ~\underset{\sigma(Z_{i-1},Z_i)}{\bot}~\sigma(Z_0,\dots,Z_{i-2}) .
\end{equation}
From the dynamic $(\ref{pdmp:dyn2})$, we have
$$\sigma(\delta_{i-1},\varepsilon_{i-1}) ~\bot~ \sigma(Z_0,\dots,Z_{i-2})\vee\sigma(Z_{i-1}) .$$
Thus, in the light of Proposition 6.8 of \cite{Kal} in the direction $\Rightarrow$, we have
\begin{equation}\label{indep07}
\sigma(\delta_{i-1},\varepsilon_{i-1}) ~\underset{\sigma(Z_{i-1})}{\bot}~\sigma(Z_0,\dots,Z_{i-2}) .
\end{equation}
Furthermore, it is easy to see that
\begin{equation}
\label{indep08}
\sigma(Z_0,\dots,Z_{i-2}) \underset{\sigma(Z_{i-1})}{\bot}~\sigma(Z_{i-1}) .
\end{equation}
Therefore, with $(\ref{indep07})$, $(\ref{indep08})$ and Proposition 6.8 of \cite{Kal} again, but in the direction $\Leftarrow$, we have
$$\sigma(Z_{i-1},\delta_{i-1},\varepsilon_{i-1}) \underset{\sigma(Z_{i-1})}{\bot}~\sigma(Z_0,\dots,Z_{i-2}) .$$
Furthermore, we have the equality of the $\sigma$-fields
$$ \sigma(Z_{i-1},Z_i,\delta_{i-1},\varepsilon_{i-1}) = \sigma(Z_{i-1},\delta_{i-1},\varepsilon_{i-1}) ,$$
since $Z_i$ is generated by $Z_{i-1}$ and the random errors $\delta_{i-1}$ and $\varepsilon_{i-1}$. Thus,
$$\sigma(Z_i)\vee\sigma(Z_{i-1},\delta_{i-1},\varepsilon_{i-1}) \underset{\sigma(Z_{i-1})}{\bot}~\sigma(Z_0,\dots,Z_{i-2}) .$$
Thus, on the strength of Proposition 6.8 of \cite{Kal} again, in the direction $\Rightarrow$, we have
$$\sigma(Z_{i-1},\delta_{i-1},\varepsilon_{i-1}) \underset{\sigma(Z_{i-1},Z_i)}{\bot}~\sigma(Z_0,\dots,Z_{i-2}) .$$
This states $(\ref{indep06})$ since the time $S_i$ is generated from $Z_{i-1}$ and the error $\delta_{i-1}$. Therefore, we prove $(\ref{indep05})$. As a consequence, plugging $(\ref{pdmp:ind2})$ and $(\ref{indep05})$ yields to
\begin{align*}
\mathbf{E}_{\nu_0}  \big[ h_1(S_i) h_2(Z_0,\dots&,Z_{n}) | \sigma(Z_{i-1},Z_i) \big]  \\
&=~ \mathbf{E}_{\nu_0} \big[ h_1(S_i) | \sigma(Z_{i-1},Z_i) \big] \mathbf{E}_{\nu_0} \big[h_2(Z_0,\dots,Z_{n}) | \sigma(Z_{i-1},Z_i) \big] ,
\end{align*}
showing the expected result. \fin


\nocite{*}
\bibliographystyle{acm}
\bibliography{pdmp_main} 

\begin{thebibliography}{10}

\bibitem{AalPHD}
{\sc Aalen, O.~O.}
\newblock {\em Statistical inference for a family of counting processes}.
\newblock ProQuest LLC, Ann Arbor, MI, 1975.
\newblock Thesis (Ph.D.)--University of California, Berkeley.

\bibitem{Aal77}
{\sc Aalen, O.~O.}
\newblock Weak convergence of stochastic integrals related to counting
  processes.
\newblock {\em Z. Wahrscheinlichkeitstheorie und Verw. Gebiete 38}, 4 (1977),
  261--277.

\bibitem{Aal78}
{\sc Aalen, O.~O.}
\newblock Nonparametric inference for a family of counting processes.
\newblock {\em Ann. Statist. 6}, 4 (1978), 701--726.

\bibitem{AalenHistory}
{\sc Aalen, O.~O., Andersen, P.~K., Borgan, {\O}., Gill, R.~D., and Keiding,
  N.}
\newblock History of applications of martingales in survival analysis.
\newblock {\em J. \'Electron. Hist. Probab. Stat. 5}, 1 (2009), 28.

\bibitem{And}
{\sc Andersen, P.~K., Borgan, {\O}., Gill, R.~D., and Keiding, N.}
\newblock {\em Statistical models based on counting processes}.
\newblock Springer Series in Statistics. Springer-Verlag, New York, 1993.

\bibitem{MR1679540}
{\sc Aven, T., and Jensen, U.}
\newblock {\em Stochastic models in reliability}, vol.~41 of {\em Applications
  of Mathematics (New York)}.
\newblock Springer-Verlag, New York, 1999.

\bibitem{Az12a}
{\sc Aza{\"{\i}}s, R., Dufour, F., and G{\'e}gout-Petit, A.}
\newblock Nonparametric estimation of the jump rate for non-homogeneous marked
  renewal processes.
\newblock {\em Preprint arXiv:1202.2211, Accepted for publication in Annales de
  l'Institut Henri Poincar{\'e}\/} (2012).

\bibitem{BERAN}
{\sc Beran, J.}
\newblock Nonparametric regression with randomly censored survival data, 1981.
\newblock Technical report, Dept. Statist. Univ. California, Berkeley.

\bibitem{MR2528336}
{\sc Chiquet, J., and Limnios, N.}
\newblock A method to compute the transition function of a piecewise
  deterministic {M}arkov process with application to reliability.
\newblock {\em Statist. Probab. Lett. 78}, 12 (2008), 1397--1403.

\bibitem{Chung}
{\sc Chung, K.~L.}
\newblock {\em A course in probability theory}, second~ed.
\newblock Academic Press, New York-London, 1974.
\newblock Probability and Mathematical Statistics, Vol. 21.

\bibitem{COMTE}
{\sc Comte, F., Gaïffas, S., and Guilloux, A.}
\newblock Adaptive estimation of the conditional intensity of marker-dependent
  counting processes.
\newblock {\em Ann. Inst. H. Poincaré Probab. Statist. 47}, 4 (2011),
  1171--1196.

\bibitem{Cox72}
{\sc Cox, D.~R.}
\newblock Regression models and life-tables.
\newblock {\em J. Roy. Statist. Soc. Ser. B 34\/} (1972), 187--220.

\bibitem{MR932943}
{\sc Dabrowska, D.~M.}
\newblock Nonparametric regression with censored survival time data.
\newblock {\em Scand. J. Statist. 14}, 3 (1987), 181--197.

\bibitem{Dav2}
{\sc Davis, M., Dempster, M., Sethi, S., and Vermes, D.}
\newblock Optimal capacity expansion under uncertainty.
\newblock {\em Advances in Applied Probability 19}, 1 (1987), 156--176.

\bibitem{Dav}
{\sc Davis, M. H.~A.}
\newblock {\em Markov models and optimization}, vol.~49 of {\em Monographs on
  Statistics and Applied Probability}.
\newblock Chapman \& Hall, London, 1993.

\bibitem{DeS}
{\sc de~Saporta, B., Dufour, F., Zhang, H., and Elegbede, C.}
\newblock Optimal stopping for the predictive maintenance of a structure
  subject to corrosion.
\newblock {\em Journal of Risk and Reliability 226 (2)\/} (2012), 169--181.

\bibitem{MAL}
{\sc Fontbona, J., Guérin, H., and Malrieu, F.}
\newblock Quantitative estimates for the long time behavior of a {PDMP}
  describing the movement of bacteria.
\newblock {\em Preprint\/} (2010).

\bibitem{HLL}
{\sc Hern{\'a}ndez-Lerma, O., and Lasserre, J.~B.}
\newblock {\em Markov chains and invariant probabilities}, vol.~211 of {\em
  Progress in Mathematics}.
\newblock Birkh\"auser Verlag, Basel, 2003.

\bibitem{subtilin}
{\sc Hu, J., Wu, W., and Sastry, S.}
\newblock {\em Modeling subtilin production in {\normalfont bacillus subtilis}
  using stochastic hybrid systems}.
\newblock In R. Alur and G.J. Pappas, editors, {\textit{Hybrid Systems:
  Computation and Control}}, number 2993 in LNCS, Springer-Verlag, Berlin,
  2004.

\bibitem{MR676128}
{\sc Jacobsen, M.}
\newblock {\em Statistical analysis of counting processes}, vol.~12 of {\em
  Lecture Notes in Statistics}.
\newblock Springer-Verlag, New York, 1982.

\bibitem{Jacobsen}
{\sc Jacobsen, M.}
\newblock {\em Point Process Theory and Applications: Marked Point and
  Piecewise Deterministic Processes}.
\newblock Probability and its Applications. Birkhäuser, Boston-Basel-Berlin,
  2006.

\bibitem{Kal}
{\sc Kallenberg, O.}
\newblock {\em Foundations of modern probability}, second~ed.
\newblock Probability and its Applications (New York). Springer-Verlag, New
  York, 2002.

\bibitem{MR1345201}
{\sc Li, G., and Doss, H.}
\newblock An approach to nonparametric regression for life history data using
  local linear fitting.
\newblock {\em Ann. Statist. 23}, 3 (1995), 787--823.

\bibitem{Martinussen2006}
{\sc Martinussen, T., and Scheike, T.~H.}
\newblock {\em Dynamic regression models for survival data}.
\newblock Statistics for Biology and Health. Springer, New York, 2006.

\bibitem{McK90}
{\sc McKeague, I.~W., and Utikal, K.~J.}
\newblock Inference for a nonlinear counting process regression model.
\newblock {\em Ann. Statist. 18}, 3 (1990), 1172--1187.

\bibitem{MandT}
{\sc Meyn, S., and Tweedie, R.~L.}
\newblock {\em Markov chains and stochastic stability}, second~ed.
\newblock Cambridge University Press, Cambridge, 2009.

\bibitem{Ouv}
{\sc Ouvrard, J.-Y.}
\newblock {\em Probabilités 2}.
\newblock Deuxième édition, Master, Agrégation. Cassini, 2004.

\bibitem{Ram83}
{\sc Ramlau-Hansen, H.}
\newblock Smoothing counting process intensities by means of kernel functions.
\newblock {\em Ann. Statist. 11}, 2 (1983), 453--466.

\bibitem{MR840519}
{\sc Stute, W.}
\newblock Conditional empirical processes.
\newblock {\em Ann. Statist. 14}, 2 (1986), 638--647.

\bibitem{MR1208878}
{\sc Utikal, K.~J.}
\newblock Nonparametric inference for a doubly stochastic {P}oisson process.
\newblock {\em Stochastic Process. Appl. 45}, 2 (1993), 331--349.

\bibitem{MR1445041}
{\sc Utikal, K.~J.}
\newblock Nonparametric inference for {M}arkovian interval processes.
\newblock {\em Stochastic Process. Appl. 67}, 1 (1997), 1--23.

\end{thebibliography}

\end{document}